\documentclass{amsart}
\usepackage{amssymb, amsmath, latexsym, amsthm}
\usepackage{stackengine}
\usepackage{pdfpages}
\stackMath
\usepackage{fancyhdr}
\usepackage{graphicx}
\usepackage{multicol}

\usepackage[margin=1in]{geometry}
\usepackage{mathtools}
\usepackage{empheq}
\usepackage[most]{tcolorbox}

\newtheoremstyle{dotless}{}{}{\itshape}{}{\bfseries}{}{ }{}
\theoremstyle{dotless}

\newtheorem{theorem}{Theorem}[section]
\newtheorem{lemma}{Lemma}
\theoremstyle{definition}

\DeclareMathOperator{\sign}{sign}

\newcommand{\La}{\langle}
\newcommand{\Ra}{\rangle}

\def\sign{\operatorname{sgn}}

\newcommand{\al}{\alpha}
\newcommand{\eps}{\varepsilon}

\newcommand{\cD}{\mathcal D}

\newcommand{\cF}{\mathcal F}

\newcommand{\cT}{\mathcal T}

\newcommand{\cU}{\mathcal U}

\newcommand{\bC}{\mathbb C}

\newcommand{\bR}{\mathbb R}

\newcommand{\bfT}{\mathbf T}

\newcommand{\hdp}{h_{J''}}

\newcommand{\one}{\mathbf{1}}

\newtheorem{remark}[theorem]{Remark}
\newtheorem{defin}[theorem]{Definition}

\begin{document}

\title[Dyadic bi-parameter repeated commutator and dyadic product BMO]%
{Dyadic bi-parameter repeated  commutator and dyadic product BMO}
\author[I.~ Holmes]{Irina Holmes}
\thanks{IH is partially supported by the NSF an NSF Postdoc under Award No.1606270}
\address{Department of Mathematics, Texas A\&M University, College Station, TX 77843-3368 }
\author[S.~Treil]{Sergei Treil}
\thanks{ST is partially supported by the NSF  DMS 1856719}
\address{Department of Mathematics, Brown University, 151 Thayer Str./Box 1917, Providence, RI 02912, USA}
\email{treil@math.brown.edu \textrm{(S. Treil)}}
\author[A.~Volberg]{Alexander Volberg}
\thanks{AV is partially supported by the NSF grant DMS 1900268 and by Alexander von Humboldt foundation}
\address{Department of Mathematics, Michigan Sate University, East Lansing, MI. 48823}
\email{volberg@math.msu.edu \textrm{(A.\ Volberg)}}
\makeatletter
\@namedef{subjclassname@10010}{
  \textup{10010} Mathematics Subject Classification}
\makeatother
\subjclass[10010]{42B100, 42B35, 47A100}
\begin{abstract} 
Let $\bfT$ is a certain tensor product of simple dyadic shifts defined below.
We prove here that for dyadic bi-parameter repeated commutator its norm can be estimated from below by Chang-Fefferman $BMO$ norm pertinent to its symbol. See Theorems \ref{mainskip}, \ref{mainskip1} at the end of this article. But this is done below under an extra assumption on the Haar--Fourier side of the symbol.  In Section \ref{analyze} we carefully analyze what goes wrong in the absence of this extra assumption. At the end of this note we also list a counterexample to the existing proof of characterization of bi-parameter repeated commutator with the Hilbert transforms. This is a counterexample to the proof, and it is not a counterexample to the statement of factorization result in bi-disc, or to Nehari's theorem in bi-disc.
\end{abstract}
\maketitle


\section{Introduction}

In \cite{HTV} we considered the characterization of the symbols of bounded commutator of multiplication and simple dyadic bi-parameter  singular integral.
In this note we are considering a  dyadic model of {\it repeated} bi-parameter commutator. It has been noticed that repeated commutators are much harder to treat than the simple ones. The former still await the understanding, see below. The latter are well understood in many situations, see \cite{FS}, \cite{DLWY},  \cite{DKPS},  \cite{HTV}.

Bi-parameter theory is notoriously more difficult than the more classical one  parameter theory of singular integrals. The good place to get acquainted with multi-parameter specifics are the papers of J.-L. Journ\'e \cite{JLJ}, \cite{JLJ2} and Muscalu--Pipher--Tao--Thiele \cite{MPTT1}, \cite{MPTT2}. The applications to analysis in polydisc can be found in \cite{Ch},  \cite{ChF}, \cite{ChF2}.

It is well known \cite{ChF, ChF2, JLJ, JLJ2} that in the multi-parameter setting all concepts of Carleson measure, $BMO$, John--Nirenberg inequality, Calder\'on--Zygmund decomposition (used in classical theory) are much more delicate. Paper \cite{MPTT1}  develops a completely new approach to prove natural tri-linear bi-parameter estimates on bi-parameter paraproducts, especially outside of Banach range. In \cite{MPTT1} Journ\'e's lemma \cite{JLJ2} was used, but the approach did not generalize to multi-parameter paraproduct forms. This issue was resolved in \cite{MPTT2}, where a simplified method was used  to address the multi-parameter paraproducts. But repeated bi-parameter commutators  involve {\it several} bi-parameter paraproducts at once. Hence, the estimate of repeated bi-parameter commutators {\it from below} becomes tough: one needs to isolate  those paraproducts one from another. And they do not  like to be isolated.

One of the new feature of the multi-parameter theory is captured by several different definitions of $BMO$, see \cite{FS}, \cite{Ch}. The necessity of those new effects was discovered first by Carleson \cite{Car}, see also \cite{Tao}. 

The difficulties with multi-parameter theory was highlighted recently by two very different series of papers. One concerns itself with the Carleson measure and Carleson embedding on polydisc and on multi-tree, see, e.g. \cite{AMPS18}, \cite{AHMV}, \cite{MPV}, \cite{MPVZ}.
These papers, roughly speaking, are devoted to harmonic analysis (Carleson embedding in particular) on graphs with cycles. This theory is drastically different from the usual one, and it is much more difficult.

Another particularity of the multi-parameter theory is highlighted by the repeated commutator story. Lacey and Ferguson \cite {FL} found the characterization of the symbols that give us bounded ``small Hankel operators". It was a breakthrough article that gave a multi-parameter Nehari theorem and  a long searched after factorization of bi-parameter Hardy space $H^1$. It was exactly equivalent to a bi-parameter ``repeated commutator characterization". Several papers followed where ``bi-'' was upgraded to ``multi-'', and where repeated commutation was performed with different classical singular integrals: \cite{DP}, \cite {FL}, \cite{L}, \cite {LPPW}, \cite {LT}.  

There are many new and beautiful ideas in the above mentioned papers devoted to this subject. We also want to mention  \cite{OPS}, where the authors use an argument inspired by Toeplitz operators to  
show the lower bounds. It assumes the lower norm for the Hilbert  
transform  claimed in \cite{FL}, \cite{LPPW} as a black  
box.  But there is a problem with \cite{FL} and with what followed.

That was a big breakthrough in multi-parameter theory. Unfortunately, \cite{V} indicated a hole in  the proofs of \cite{FL}, \cite{L}, and this circle of problems is still unsettled. See also Section \ref{firstex} for the more detailed explanation.

We decided to simplify the problem to the ``bare bones", and to consider repeated commutators with the simplest dyadic singular integral, namely, with the dyadic shift of order $1$. The repeated commutator problem presented the same type of difficulty as in \cite{FL} and subsequent papers. But we analyze these difficulties below. They become very clear (but still largely unsurmountable) in the dyadic case we consider below. 

However, with an extra assumption on the symbol of repeated commutator we managed  to get the expected answer, see Theorems \ref{mainskip}, \ref{mainskip1} at the end of this article. The expected answer is of course that the repeated commutator is bounded if and only if the symbol belongs to dyadic Chang--Fefferman $BMO$. If symbol does belong to dyadic Chang--Fefferman $BMO$ then the repeated commutator is indeed bounded, see e.g. \cite{LM}.

In the opposite direction the only fully proved result so far is in \cite{BP}, but this paper requires that infinitely many commutators be bounded to conclude that the symbol is bounded.

As we already mentioned with an extra assumption on the symbol we can prove the necessity of  dyadic Chang--Fefferman $BMO$ (see Theorems \ref{mainskip}, \ref{mainskip1} at the end of this article).
We list carefully what are the difficulties to be surmounted in the general case (without an extra assumption on the symbol), and what kind of combinatorial morass should be taken care of. This is left to a possibly more skillful reader.

\section{Repeated commutator}
\label{S:Intro}

Let $\cD$ be the usual dyadic lattice on the line and $\cD\times \cD$ be the family of dyadic rectangles on the plane.
Haar functions are denote $\{h_I\}_{I \in \cD}$. We consider $R=I\times J$ and $h_R= h_I\otimes h_J$.
Let $T_1$ be the dyadic shift defined by on the functions of variable $x$:
$$
T_1 h_I = h_{I_+}- h_{I_-}, \,\, \text{if}\,\, I\,\, \text{is even}; \quad T_1h_I=0,\,\, \text{otherwise}\,.
$$
Let $T_2$ be the copy of $T_1$, but acting on an independent variable $y$:
$$
T_2 h_J = h_{J_+}- h_{J_-}, \,\, \text{if}\,\, J\,\, \text{is even}; \quad T_2h_J=0,\,\, \text{otherwise}\,.
$$
Test functions are functions $f= f(x, y)$, and $\bfT := T_1\otimes T_2$.

We are interested in the commutator
$$
\cT^b:= [T_2, [T_1, b]] =T_2 (T_1 b - b T_1) - (T_1 b - b T_1)T_2 = T_2T_1 b - T_2 b T_1 - T_1b T_2 + b T_1T_2
$$
and in characterization of its boundedness. The ultimate goal is to prove that its boundedness is equivalent to

$$
b\in BMO^d,
$$
where $BMO$ is understood in Chang--Fefferman definition. We always drop the superscript $d$ as everything is  dyadic below.

\bigskip

Simple commutator $[T_1, b]$ can be naturally split to $3$ parts: $D$, $\pi$, $Z$ parts. They correspond to different parts of the product operator $f\to b\cdot f$. 

First $D$ part follows.
$$
(b, f) \to D(b, f):=\sum_I \sum_{I': I\subsetneq I'} \Delta_{I'} b \cdot\Delta_I f=
$$
$$
\sum_I (b, \one_I/|I|) (f, h_I) h_I\,.
$$
Now compute commutator:
$$
TD(b, f) - D(b, Tf)  = \sum_I (b, \one_I/|I|) (f, h_I) Th_I - \sum_I (b, \one_I/|I|) (Tf, h_I) h_I\,.
$$
Now compute commutator again:
$$
TD(b, f) - D(b, Tf)  = \sum_I \La b \Ra_I ( T(\Delta_I f) - \Delta_I(Tf))\,.
$$
So let us consider
$$
 T(\Delta_I f) - \Delta_I(Tf)
 $$
 for $f= h_K$.  Let $K=I$ and let $I$ be even (meaning $|I|= 2^{-2k}$). Then
 $T(\Delta_I h_K) = T h_I = h_{I_+}- h_{I_-}$, and $\Delta_I (Th_K) = \Delta_I (Th_I) =0$.
 
 Let now $K=I$ and let $I$  be odd. Then $T(\Delta_I h_K) = T h_I = 0$,  and $\Delta_I (Th_K) = \Delta_I (Th_I) =\Delta_I ( h_{I_+}- h_{I_-}) =0$.
 
 We conclude that $TD(b, f) - D(b, Tf) $ will have terms: 
 $$
 \sum_{I\, even} \La b \Ra_I (f, h_I)( h_{I_+}- h_{I_-})= \sum_{\hat I\, even}  \La b \Ra_{\hat I} (f, h_{\hat I}) s(I, \hat I) h_I =\sum_{ I\, odd}  \La b \Ra_{\hat I} (f, h_{\hat I}) s(I, \hat I) h_I,
 $$
 where 
 \begin{equation}
 \label{sI}s(I, \hat I) = \begin{cases} 1 \quad if \,\,I= (\hat I)_+\\
  -1  \quad otherwise.
  \end{cases}
  \end{equation}
 There will be more terms.
 Denote by $s(I)$ the dyadic sibling of $I$. Denote by $\hat I$ the immediate predecessor of $I$. Let $K=\hat I$ and let $\hat I$  be even.
 Then $\Delta_I h_K =0$, so $T(\Delta_I h_K)=0$, but $Th_K = Th_{\hat I} = h_{(\hat I)_+} - h_{(\hat I)_-}$
 and  so $Th_K= \pm( h_I- h_{s(I)})$, and, thus, 
$\Delta_I Th_K = \pm h_I$. 
So  for $f=h_K$ with $K=\hat I$ and $\hat I$ being even  we have
$$ 
\text{On}\,\, I:\,\,T(\Delta_I f) - \Delta_I(Tf) = - h_I, \,\,\text{if}\,\, I= (\hat I)_+; \quad  T(\Delta_I f) - \Delta_I(Tf) =+ h_I \,\,\text{if}\,\, I= (\hat I)_-\,.
$$
 We conclude that $TD(b, f) - D(b, Tf) $ will have also the following terms: 
 $$
 -\sum_{\hat I\, even} \La b \Ra_{ I}(f, h_{\hat I}) s(I, \hat I) h_I = -\sum_{ I\, odd} \La b \Ra_{ I}(f, h_{\hat I}) s(I, \hat I) h_I\,.
  $$
  All other $h_K$ give zero contribution.
  
  We get finally that 
  \begin{equation}
  \label{DT}
   (T_1 D_b - D_b T_1)(f)=T_1D(b, f) - D(b, T_1f)  = \sum_{I\, odd} ( \La b \Ra_{\hat I} -  \La b \Ra_{I}) (f, h_{\hat I}) s(I, \hat I) h_I
   = \sum_{I\, odd} \frac{-1}{|\hat I|^{1/2}} (b, h_{\hat I})(f, h_{\hat I})  h_I \,.
\end{equation}

\bigskip

Now we need to commute with $T_2$:
$$
[T_2, [T_1, D_b]](f) =\sum_{I odd} \frac{h_I}{|I|^{1/2}} \Big(T_2\{ (b, h_{\hat I})(f, h_{\hat I})\} -  (b, h_{\hat I})(T_2f, h_{\hat I})\Big)\,.
$$
We use now notations:
$b_I := (b, h_I)\,.$ Then we rewrite the previous display as follows:

$$
[T_2, [T_1, D_b]](f) =\sum_{I odd} \frac{h_I}{|I|^{1/2}} (T_2 b_{\hat I}- b_{\hat I} T_2) (f_{\hat I})\,.
$$
Now multiplication in the second variable $f_{\hat I}\to b_{\hat I} \cdot f_{\hat I}$ can be decomposed to $D, \pi, Z$ parts.

\subsection{$DD$ part of repeated commutator}
$$
DD=D_2 ([T_2, [T_1, D_b]](f)) = \sum_{I odd} \frac{h_I}{|I|^{1/2}} D_{part}(T_2 b_{\hat I}- b_{\hat I} T_2)(f_{\hat I})=    \sum_{I odd} \frac{h_I}{|I|^{1/2}}\sum_{J odd} \frac{h_J}{|J|^{1/2}}(b_{\hat I}, h_{\hat J})(f_{\hat I}, h_{\hat J})=
$$
$$
 \sum_{I, J odd, odd} (b, h_{\hat I} \otimes h_{\hat J}) (f, h_{\hat I} \otimes h_{\hat J}) \frac{h_I(x) \cdot h_J(y)}{|I|^{1/2} |J|^{1/2}}=
$$
$$
 \sum_{R odd, odd} (b, h_{\hat R}) (f, h_{\hat R}) \frac{h_R}{|R|^{1/2}}\,.
$$

\subsection{$D\pi$ part of repeated commutator}
$$
D\pi = \pi_2([T_2, [T_1, D_b]](f)) = \sum_{I odd} \frac{h_I}{|I|^{1/2}} \pi_{part}(T_2 b_{\hat I}- b_{\hat I} T_2)(f_{\hat I})= \sum_{I odd} \frac{h_I}{|I|^{1/2}} (T_2 \pi_{b_{\hat I}}- \pi_{b_{\hat I}} T_2) (f_{\hat I})
$$
Let us recall what is  $\pi_b$.
$$
(b_{\hat I}, f_{\hat I}) \to \pi_{b_{\hat I}} f_{\hat I}:=\sum_{J} \sum_{J': J\subsetneq J'} \Delta_{J} b_{\hat I} \cdot\Delta_{J'} f_{\hat I} = \sum_{J} \La f_{\hat I}\Ra_J (b_{\hat I}, h_J) h_J\,.
$$
Now compute commutator (below $\tilde\one_J := \one_J/|J|$).
\begin{align}
\label{Dpi}
&D\pi= \sum_{I odd} \frac{h_I}{|I|^{1/2}}  \Big(T_2\pi_{b_{\hat I}} f_{\hat I} - \pi_{b_{\hat I}} (T_2 f_{\hat I})\Big)  = 
\\
&\sum_{I\, odd}  \sum_{ J even} (b, h_{\hat I}\otimes h_J) (f, h_{\hat I}\otimes \tilde \one_J) \frac{h_I}{|I|^{1/2}} \otimes T_2 h_J - \notag
\\
&\sum_{I\, odd} \sum_{ J}(b, h_{\hat I}\otimes h_J) (T_2 f, h_{\hat I}\otimes \tilde \one_J) \frac{h_I}{|I|^{1/2}}  \otimes  h_J\,.  \notag
\end{align}

Let us use the following notations for given $I, J$:
\begin{equation}
\label{nota1}
R=I\times J, \hat R = \hat I\times \hat J, \hat R_1=  \hat I\times  J,  \hat R_2 = I\times \hat J, h_R = h_I\otimes h_J
\end{equation}
\begin{equation}
\label{nota2}
\tilde \one_R = \tilde \one_I\otimes \tilde\one_J, h\one_R  = h_I\otimes \tilde\one_J, \one h_R= \tilde\one_I\otimes h_J
\end{equation}

\subsection{$DZ$ part of repeated commutator}

$$
DZ= Z_2([T_2, [T_1, D_b]](f)) = \sum_{I odd} \frac{h_I}{|I|^{1/2}} Z_{part}(T_2 b_{\hat I}- b_{\hat I} T_2)(f_{\hat I})= \sum_{I odd} \frac{h_I}{|I|^{1/2}} (T_2Z_{b_{\hat I}}- Z_{b_{\hat I}} T_2) (f_{\hat I})
$$
Let us recall what is  $Z_b$.
$$
(b_{\hat I}, f_{\hat I}) \to Z_{b_{\hat I}} f_{\hat I}:= \sum_{J}  (b_{\hat I}, h_J) ( f_{\hat I}, h_J) \tilde\one_J \,.
$$

Therefore,

\begin{align}
\label{DZ}
&DZ =  \sum_{I odd} \frac{h_I}{|I|^{1/2}} \Big(T_2Z_{b_{\hat I}}(f_{\hat I})- Z_{b_{\hat I}} (T_2f_{\hat I})\Big)=
\\
&\sum_{I\, odd} \sum_{ J even} (b, h_{\hat I}\otimes h_J) (f, h_{\hat I}\otimes h_J) \frac{h_I}{|I|^{1/2}} \otimes T_2 \tilde \one_J -\notag
\\
&\sum_{I\, odd}  \sum_{ J}(b, h_{\hat I}\otimes h_J) (T_2 f, h_{\hat I}\otimes,  h_J)\frac{h_I}{|I|^{1/2}} \otimes  \tilde \one_J\,.\notag
\end{align}

This finishes $DD, D\pi, DZ$ parts of repeated commutator.

Now we need to write $\pi D, Z\pi$ parts--they are just symmetric with respect to exchange the first and the second variables. But also we need $ZZ$ part, $\pi Z, Z\pi$ parts, and the most dangerous $\pi\pi$ part.

\subsection{$\pi$-parts}
Here it is.
$$
T_2 (T_1 \pi_b f - \pi_b (T_1 f)) - (T_1\pi_b (T_2 f) - \pi_b (T_1T_2 f)) =
$$
$$
= \big(T_2 T_1 \pi_b f  - T_1\pi_b (T_2 f)\big)-  \big( T_2\pi_b (T_1 f) -  \pi_b (T_1T_2 f)\big)=
$$
$$
\sum_I (T_2  b_I - b_I T_2) \La f \Ra_I T_1 h_I - \sum_I (T_2 b_I - b_I T_2)\La T_1 f\Ra_I h_I\,.
$$
Its $\pi D$ part is obtained  by calculating $D$ parts of commutator $(T_2  b_I - b_I T_2)$. So as before
\begin{align}
\label{piD}
&\pi D=
\\
&\sum_{I even}  \sum_{ J odd} (b, h_{ I}\otimes h_{\hat J}) (f,  \tilde \one_I \otimes h_{\hat J}) T_1 h_I \otimes \frac{h_J}{|J|^{1/2}}  -   \notag
\\
&\sum_{I} \sum_{ J odd}(b, h_{ I}\otimes h_{\hat J}) (T_1 f,  \tilde \one_I \otimes h_{J}\otimes) h_I \otimes\frac{h_J}{|J|^{1/2}}  \,.   \notag
\end{align}
We see that \eqref{piD} is symmetric to \eqref{Dpi}, as it should be.

\begin{align}
\label{piZ}
&\pi Z= 
\\
& \sum_{I even}\sum_J (b, h_I\otimes h_J) (f, \tilde\one_I\otimes h_J) T_1 h_I\otimes T_2 \tilde\one_J - \notag
\\
&\sum_{I even}\sum_J (b, h_I\otimes h_J) (T_2 f, \tilde\one_I\otimes h_J) T_1 h_I\otimes  \tilde\one_J -\notag
\\
& \sum_{I even}\sum_J (b, h_I\otimes h_J) (T_1 f, \tilde\one_I\otimes h_J) h_I\otimes T_2 \tilde\one_J +\notag
\\
& \sum_{I even}\sum_J (b, h_I\otimes h_J) (T_1T_2f, \tilde\one_I\otimes h_J)  h_I\otimes  \tilde\one_J\,.\notag
\end{align}

\bigskip

Now the most dangerous $\pi\pi$ part of repeated commutator follows:

\begin{align}
\label{pipi}
&\pi\pi= 
\\
&\sum_{I, J even, even} (b, h_I\otimes h_J) (f, \tilde\one_I\otimes \tilde\one_J) T_1 h_I\otimes T_2 h_J -\notag
\\
&\sum_{I any, J even} (b, h_I\otimes h_J) (T_1f, \tilde\one_I\otimes \tilde\one_J)  h_I\otimes T_2 h_J-\notag
\\
&\sum_{I  even, J any} (b, h_I\otimes h_J) (T_2f, \tilde\one_I\otimes \tilde\one_J) T_1 h_I\otimes  h_J +\notag
\\
&\sum_{I any , J any } (b, h_I\otimes h_J) (T_1T_2f, \tilde\one_I\otimes \tilde\one_J)  h_I\otimes  h_J\,.\notag
\end{align}

\subsection{$Z$-parts}

\begin{align}
\label{ZD}
&ZD =
\\
&\sum_{I} \sum_{J odd} (b, h_I\otimes h_{\hat J} ) (f, h_I\otimes h_{\hat J}) T_1 \tilde\one_I\otimes h_J /|J|^{1/2} - \notag
\\
&\sum_{I} \sum_{J odd} (b, h_I\otimes h_{\hat J} ) (T_1 f, h_I\otimes h_{\hat J})  \tilde\one_I\otimes h_J/J|^{1/2}\,.\notag
\end{align}

\begin{align}
\label{Zpi}
&Z\pi= 
\\
& \sum_{I even}\sum_J (b, h_I\otimes h_J) (f,  h_I\otimes \tilde\one_J) T_1  \tilde\one_I\otimes T_2 h_J - \notag
\\
&\sum_{I even}\sum_J (b, h_I\otimes h_J) (T_1 f,  h_I\otimes \tilde\one_J)   \tilde\one_I\otimes T_2 h_J -\notag
\\
& \sum_{I even}\sum_J (b, h_I\otimes h_J) (T_2 f,  h_I\otimes \tilde\one_J) T_1  \tilde\one_I\otimes  h_J+\notag
\\
& \sum_{I even}\sum_J (b, h_I\otimes h_J) (T_1T_2f,  h_I\otimes \tilde\one_J)   \tilde\one_I\otimes  h_J\,.\notag
\end{align}

\begin{align}
\label{ZZ}
&ZZ= 
\\
& \sum_{I even}\sum_J (b, h_I\otimes h_J) (f,  h_I\otimes h_J) T_1  \tilde\one_I\otimes T_2  \tilde\one_J - \notag
\\
&\sum_{I even}\sum_J (b, h_I\otimes h_J) (T_1 f,  h_I\otimes h_J)   \tilde\one_I\otimes T_2  \tilde\one_J -\notag
\\
& \sum_{I even}\sum_J (b, h_I\otimes h_J) (T_2 f, h_I\otimes h_J) T_1  \tilde\one_J\otimes   \tilde\one_J+\notag
\\
& \sum_{I even}\sum_J (b, h_I\otimes h_J) (T_1T_2f,  h_I\otimes h_J)   \tilde\one_J\otimes   \tilde\one_J\,.\notag
\end{align}

\newpage

\section{Necessity of $BMO_r$}
\label{BMOr}

In this section we show that if $\cT^b$ is bounded, then $b \in BMO_r$.
To do this, we test $\cT^b$ on a Haar function $h_{I''}\otimes h_{J''}$. This yields 25 operators, which we list below. Most paraproducts in Section \ref{S:Intro} (like $\pi\pi$ for example) yield 2-4 terms, so names such as $\pi\pi1$ below simply indicate that the term comes from $\pi\pi$. The operators below also have various restrictions on the parity of $I''$ and $J''$, i.e. most only appear for a certain combination of parities.  This is indicated in parentheses next to each operator.  

The operators are grouped by simplicity: Group I are the simplest operators containing finitely many terms, Groups IIa and IIb are product operators, and lastly the four operators in Group III are the most difficult. They are also color-coded...

\begin{center}
\textbf{Group I: Simple operators}
\end{center}

\begin{multicols}{2}
\begin{tcolorbox}[ams align*,colback=white!75!red]
& \boxed{\bf{DZ2}} \:\text{\bf ($I''$ even, $J''$ even)}\\
& (b, h_{I''}\otimes h_{J''_+}) \frac{1}{\sqrt{|I''|}}(h_{I''_+} + h_{I''_-}) \tilde\one_{J''_+}\\
 - & (b, h_{I''}\otimes h_{J''_-}) \frac{1}{\sqrt{|I''|}}(h_{I''_+} + h_{I''_-}) \tilde\one_{J''_-}
\end{tcolorbox}

\begin{tcolorbox}[ams align*,colback=white!75!red]
& \boxed{\bf{DD}} \:\text{($I''$ even, $J''$ even)}\\
& (b, h_{I''}\otimes h_{J''}) \frac{1}{\sqrt{|I''||J''|}} (h_{I''_+} + h_{I''_-}) (h_{J''_+} + h_{J''_-})
\end{tcolorbox}

\begin{tcolorbox}[ams align*,colback=white!75!red]
& \boxed{\bf{ZD2}} \:\text{($I''$ even, $J''$ even)}\\
& (b, h_{I''_+}\otimes h_{J''}) \frac{1}{\sqrt{|J''|}} \tilde\one_{I''_+} (h_{J''_+} + h_{J''_-})\\
 - & (b, h_{I''_-}\otimes h_{J''}) \frac{1}{\sqrt{|J''|}} \tilde\one_{I''_-} (h_{J''_+} + h_{J''_-})
\end{tcolorbox}

\begin{tcolorbox}[ams align*,colback=white!75!red]
& \boxed{\bf{ZZ4}} \:\text{($I''$ even, $J''$ even)}\\
& (b, h_{I''_+}\otimes h_{J''_+}) \tilde\one_{I''_+} \tilde\one_{J''_+}
	- (b, h_{I''_+}\otimes h_{J''_-}) \tilde\one_{I''_+} \tilde\one_{J''_-}\\
 - & (b, h_{I''_-}\otimes h_{J''_+}) \tilde\one_{I''_-} \tilde\one_{J''_+}
	+ (b, h_{I''_-}\otimes h_{J''_-}) \tilde\one_{I''_-} \tilde\one_{J''_-}
\end{tcolorbox}
\end{multicols}

\begin{multicols}{2}
\begin{tcolorbox}[ams align*,colback=white!75!green]
& \boxed{\bf{ZD1}} \:\text{\bf ($J''$ even)}\\
& -(b, h_{I''}\otimes h_{J''}) \frac{1}{\sqrt{|J''|}} (T_1 \tilde\one_{I''}) (h_{J''_+} + h_{J''_-})
\end{tcolorbox}

\begin{tcolorbox}[ams align*,colback=white!75!blue]
& \boxed{\bf{DZ1}} \:\text{\bf ($I''$ even)}\\
& -(b, h_{I''}\otimes h_{J''}) \frac{1}{\sqrt{|I''|}} (h_{I''_+} + h_{I''_-}) (T_2 \tilde\one_{J''})
\end{tcolorbox}
\end{multicols}

\begin{multicols}{2}
\begin{tcolorbox}[ams align*,colback=white!75!green]
& \boxed{\bf{ZZ2}} \:\text{\bf ($J''$ even)}\\
& - (b, h_{I''}\otimes h_{J''_+}) (T_1 \tilde\one_{I''}) \tilde\one_{J''_+}\\
& + (b, h_{I''}\otimes h_{J''_-}) (T_1 \tilde\one_{I''}) \tilde\one_{J''_-}
\end{tcolorbox}

\begin{tcolorbox}[ams align*,colback=white!75!blue]
& \boxed{\bf{ZZ3}} \:\text{\bf ($I''$ even)}\\
& - (b, h_{I''_+}\otimes h_{J''}) \tilde\one_{I''_+} (T_2 \tilde\one_{J''}) \\
& + (b, h_{I''_-}\otimes h_{J''}) \tilde\one_{I''_-} (T_2 \tilde\one_{J''}) 
\end{tcolorbox}
\end{multicols}

\begin{center}
\begin{tcolorbox}[ams align*,colback=white!75!orange]
& \boxed{\bf{ZZ1}} \:\text{\bf (no restrictions)}\\
& (b, h_{I''}\otimes h_{J''}) (T_1 \tilde\one_{I''}) (T_2 \tilde\one_{J''})
\end{tcolorbox}
\end{center}

\begin{center}
\textbf{Group IIa: Product operators, summation over $I$}
\end{center}

\begin{tcolorbox}[ams align*,colback=white!75!red]
& \boxed{\bf{\pi D2}} \:\text{\bf ($I''$ even, $J''$ even)}\\
& \left[ \sum_{I\subsetneq I''_+} (b, h_I\otimes h_{J''}) h_{I''_+}(I) h_I \right]\frac{h_{J''_+} + h_{J''_-}}{\sqrt{|J''|}}
 - \left[ \sum_{I\subsetneq I''_-} (b, h_I\otimes h_{J''}) h_{I''_-}(I) h_I \right]\frac{h_{J''_+} + h_{J''_-}}{\sqrt{|J''|}} 
\end{tcolorbox}

\begin{tcolorbox}[ams align*,colback=white!75!red]
& \boxed{\bf{\pi Z2}} \:\text{\bf ($I''$ even, $J''$ even)}\\
&   \left[ \sum_{I\subsetneq I''_+} (b, h_I\otimes h_{J''_+}) h_{I''_+}(I) h_I \right] \tilde\one_{J''_+}
 - \left[ \sum_{I\subsetneq I''_+} (b, h_I\otimes h_{J''_-}) h_{I''_+}(I) h_I \right] \tilde\one_{J''_-}\\ 
& - \left[ \sum_{I\subsetneq I''_-} (b, h_I\otimes h_{J''_+}) h_{I''_-}(I) h_I \right] \tilde\one_{J''_+}
 + \left[ \sum_{I\subsetneq I''_-} (b, h_I\otimes h_{J''_-}) h_{I''_-}(I) h_I \right] \tilde\one_{J''_-}
\end{tcolorbox}

\begin{multicols}{2}
\begin{tcolorbox}[ams align*,colback=white!75!green]
& \boxed{\bf{\pi D1}} \:\text{\bf ($J''$ even)}\\
& - \left[ \sum_{\substack{I\subsetneq I''\\ I \text{ even}}} (b, h_I\otimes h_{J''}) h_{I''}(I) (T_1h_I) \right]\frac{h_{J''_+} + h_{J''_-}}{\sqrt{|J''|}}
\end{tcolorbox}

\begin{tcolorbox}[ams align*,colback=white!75!blue]
& \boxed{\bf{\pi Z1}} \:\text{\bf ($I''$ even)}\\
& - \left[ \sum_{I\subsetneq I''_+} (b, h_I\otimes h_{J''}) h_{I''_+}(I) h_I \right] (T_2 \tilde\one_{J''})\\
& + \left[ \sum_{I\subsetneq I''_-} (b, h_I\otimes h_{J''}) h_{I''_-}(I) h_I \right] (T_2 \tilde\one_{J''})
\end{tcolorbox}

\begin{tcolorbox}[ams align*,colback=white!75!green]
& \boxed{\bf{\pi Z4}} \:\text{\bf ($J''$ even)}\\
& - \left[ \sum_{\substack{I\subsetneq I''\\ I \text{ even}}} (b, h_I\otimes h_{J''_+}) h_{I''}(I) (T_1h_I) \right]
	\tilde\one_{J''_+}\\
& + \left[ \sum_{\substack{I\subsetneq I''\\ I \text{ even}}} (b, h_I\otimes h_{J''_-}) h_{I''}(I) (T_1h_I) \right]
	\tilde\one_{J''_-}
\end{tcolorbox}

\begin{tcolorbox}[ams align*,colback=white!75!orange]
& \boxed{\bf{\pi Z3}} \:\text{\bf (no restrictions)}\\
& \left[ \sum_{\substack{I\subsetneq I''\\ I \text{ even}}} (b, h_I\otimes h_{J''}) h_{I''}(I) (T_1h_I) \right]
	(T_2\tilde\one_{J''}).
\end{tcolorbox}
\end{multicols}

\begin{center}
\textbf{Group IIb: Product operators, summation over $J$}
\end{center}

\begin{tcolorbox}[ams align*,colback=white!75!red]
& \boxed{\bf{D\pi 2}} \:\text{\bf ($I''$ even, $J''$ even)}\\
&  \frac{h_{I''_+} + h_{I''_-}}{\sqrt{|I''|}} 
	\left[ \sum_{J\subsetneq J''_+} (b, h_{I''}\otimes h_J) h_{J''_+}(J) h_J \right]
 - \frac{h_{I''_+} + h_{I''_-}}{\sqrt{|I''|}} 
	\left[ \sum_{J\subsetneq J''_-} (b, h_{I''}\otimes h_J) h_{J''_-}(J) h_J \right]
\end{tcolorbox}

\begin{tcolorbox}[ams align*,colback=white!75!red]
& \boxed{\bf{Z\pi4}} \:\text{\bf ($I''$ even, $J''$ even)}\\
&  \tilde\one_{I''_+} \left[ \sum_{J\subsetneq J''_+} (b, h_{I''_+}\otimes h_J) h_{J''_+}(J) h_J \right] 
 - \tilde\one_{I''_-} \left[ \sum_{J\subsetneq J''_+} (b, h_{I''_-}\otimes h_J) h_{J''_+}(J) h_J \right] \\ 
 - & \tilde\one_{I''_+} \left[ \sum_{J\subsetneq J''_-} (b, h_{I''_+}\otimes h_J) h_{J''_-}(J) h_J \right] 
 + \tilde\one_{I''_-} \left[ \sum_{J\subsetneq J''_-} (b, h_{I''_-}\otimes h_J) h_{J''_-}(J) h_J \right] 
\end{tcolorbox}

\begin{multicols}{2}
\begin{tcolorbox}[ams align*,colback=white!75!blue]
& \boxed{\bf{ D\pi1}} \:\text{\bf ($I''$ even)}\\
& - \frac{h_{I''_+} + h_{I''_-}}{\sqrt{|I''|}}
 \left[ \sum_{\substack{J\subsetneq J''\\ J \text{ even}}} (b, h_{I''}\otimes h_J) h_{J''}(J) (T_2h_J) \right]
\end{tcolorbox}

\begin{tcolorbox}[ams align*,colback=white!75!green]
& \boxed{\bf{Z\pi2}} \:\text{\bf ($J''$ even)}\\
& - (T_1 \tilde\one_{I''}) \left[ \sum_{J\subsetneq J''_+} (b, h_{I''}\otimes h_J) h_{J''_+}(J) h_J \right] \\
& + (T_1 \tilde\one_{I''}) \left[ \sum_{J\subsetneq J''_-} (b, h_{I''}\otimes h_J) h_{J''_-}(J) h_J \right]
\end{tcolorbox}

\begin{tcolorbox}[ams align*,colback=white!75!blue]
& \boxed{\bf{Z\pi3}} \:\text{\bf ($I''$ even)}\\
& - \tilde\one_{I''_+}
	\left[ \sum_{\substack{J\subsetneq J''\\ J \text{ even}}} (b, h_{I''_+}\otimes h_J) h_{J''}(J) (T_2h_J) \right]\\
+ & \tilde\one_{I''_-}
	\left[ \sum_{\substack{J\subsetneq J''\\ J \text{ even}}} (b, h_{I''_-}\otimes h_J) h_{J''}(J) (T_2h_J) \right]
\end{tcolorbox}

\begin{tcolorbox}[ams align*,colback=white!75!orange]
& \boxed{\bf{Z\pi1}} \:\text{\bf (no restrictions)}\\
& (T_1\tilde\one_{I''})
	\left[ \sum_{\substack{J\subsetneq J''\\ J \text{ even}}} (b, h_{I''}\otimes h_J) h_{J''}(J) (T_2h_J) \right].
\end{tcolorbox}
\end{multicols}

\begin{center}
\textbf{Group III: The difficult operators}
\end{center}

\begin{tcolorbox}[ams align*,colback=white!75!red]
& \boxed{\bf{\pi\pi 4}} \:\text{\bf ($I''$ even, $J''$ even)}\\
&  \sum_{\substack{I\subsetneq I''_+ \\ J \subsetneq J''_+}} 
		(b, h_I\otimes h_J) h_{I''_+}(I) h_{J''_+}(J) h_I\otimes h_J
	- \sum_{\substack{I\subsetneq I''_+ \\ J \subsetneq J''_-}} 
		(b, h_I\otimes h_J) h_{I''_+}(I) h_{J''_-}(J) h_I\otimes h_J\\
- & \sum_{\substack{I\subsetneq I''_- \\ J \subsetneq J''_+}} 
		(b, h_I\otimes h_J) h_{I''_-}(I) h_{J''_+}(J) h_I\otimes h_J
	+ \sum_{\substack{I\subsetneq I''_- \\ J \subsetneq J''_-}} 
		(b, h_I\otimes h_J) h_{I''_-}(I) h_{J''_-}(J) h_I\otimes h_J
\end{tcolorbox}

\begin{tcolorbox}[ams align*,colback=white!75!blue]
& \boxed{\bf{\pi\pi3}} \:\text{\bf ($I''$ even)}\\
& - \sum_{\substack{I\subsetneq I''_+ \\ J \text{ even}, J \subsetneq J''}} (b, h_I\otimes h_J) 
		h_{I''_+}(I) h_{J''}(J) h_I \otimes T_2h_J
+ \sum_{\substack{I\subsetneq I''_- \\ J \text{ even}, J \subsetneq J''}} (b, h_I\otimes h_J) 
		h_{I''_-}(I) h_{J''}(J) h_I \otimes T_2h_J
\end{tcolorbox}

\begin{tcolorbox}[ams align*,colback=white!75!green]
& \boxed{\bf{\pi\pi2}} \:\text{\bf ($J''$ even)}\\
& - \sum_{\substack{I \text{ even}, I\subsetneq I'' \\ J \subsetneq J''_+}} (b, h_I\otimes h_J) 
		h_{I''}(I) h_{J''_+}(J) T_1h_I \otimes h_J
+ \sum_{\substack{I \text{ even}, I\subsetneq I'' \\ J \subsetneq J''_-}} (b, h_I\otimes h_J) 
		h_{I''}(I) h_{J''_-}(J) T_1h_I \otimes h_J
\end{tcolorbox}

\begin{tcolorbox}[ams align*,colback=white!75!orange]
& \boxed{\bf{\pi\pi1}} \:\text{\bf (no restrictions)}\\
& - \sum_{\substack{I \text{ even}, I\subsetneq I'' \\ J \text{ even}, J \subsetneq J''}} (b, h_I\otimes h_J) 
		h_{I''}(I) h_{J''}(J) T_1h_I \otimes T_2h_J
\end{tcolorbox}

The lemma below will bound most of these operators:

\begin{lemma}
\label{L:TbonP}
If $\cT^b$ is $L^2$-bounded, then the (one-parameter) commutators
	$$\frac{1}{\sqrt{|J|}} [T_1, b_J] \:\:\:\text{ and }\:\:\: \frac{1}{\sqrt{|I|}}[T_2, b_I]$$
are uniformly bounded, for any $I, J \in \cD$. Consequently,
	\begin{equation}
	\label{E:1parBMOx}
	\sum_{I' \subset I} \left|(b, h_{I'}\otimes h_J)\right|^2 \lesssim |I||J|,
	\end{equation}
and
	\begin{equation}
	\label{E:1parBMOy}
	\sum_{J' \subset J} \left|(b, h_{I}\otimes h_{J'})\right|^2 \lesssim |I||J|,
	\end{equation}	
for all $I, J \in \cD$.
\end{lemma}

\begin{proof}
We prove the result in the first parameter $x$, the other being symmetric. We test $\cT^b$ on a function
$f(x, y) = p(x) \otimes h_{J''}(y)$, where $p\in L^2(\bR)$ and $J'' \in \cD$. Note that if $\cT^b$ is bounded,
$\|\cT^b f\|_2^2 \lesssim \|f\|_2^2 = \|p\|_2^2$. 

Computing $\cT^b(p\otimes \hdp)$ yields five operators: two with no restrictions on $J''$:
	\begin{eqnarray*}
	{\tt P_1} &=& \sum_{\substack{J \text{ even,} \\ J \subsetneq J''}} \left( [T_1, b_J]p \right) 
			\hdp(J) (T_2 h_J),\\
	{\tt P_2} &=& \left([T_1, b_{J''}]p\right) (T_2 \tilde\one_{J''}),
	\end{eqnarray*}
and three which are only non-zero if $J''$ is \textit{even}:
	\begin{eqnarray*}
	{\tt P_3} &=& - \left([T_1, b_{J''}]p\right) \frac{1}{\sqrt{|J''|}} (h_{J''_{-}} + h_{J''_{+}}),\\
	{\tt P_4} &=& - \sum_{J \subsetneq J''_+} \left([T_1, b_J]p\right) h_{J''_+}(J) h_J +
			\sum_{J \subsetneq J''_-} \left([T_1, b_J]p\right) h_{J''_-}(J) h_J,\\
	{\tt P_5} &=& - \left([T_1, b_{J''_+}]p\right) \tilde\one_{J''_+}
				+ \left([T_1, b_{J''_-}]p\right) \tilde\one_{J''_-}.
	\end{eqnarray*}

Suppose first that $J''$ is \textit{even}. In this case, we have all five operators above. We look at the Haar coefficients of each:

\begin{equation*}
\left({\tt P_1}, h_{I_0}\otimes h_{J_0}\right) = \left\{ \begin{array}{rl}
	\left([T_1, b_{J_0}]p, h_{I_0}\right) s(J_0, \hat J_0) \hdp (\hat J_0), & \text{ if } J_0 \text{ is odd, and }
		\hat J_0 \subsetneq J''\\
	0, & \text{ otherwise.}
\end{array}\right.
\end{equation*}

\begin{equation*}
\left({\tt P_2}, h_{I_0}\otimes h_{J_0}\right) = \left\{ \begin{array}{rl}
	\left([T_1, b_{J''}]p, h_{I_0}\right) s(J_0, \hat J_0) h_{\hat J_0}(J''),  & \text{ if } J_0
		\text{ is odd, and } \hat J_0\supsetneq J''	\\
	0, & \text{ otherwise.}
\end{array}\right.
\end{equation*}

\noindent Since $J''$ is even, the inclusion conditions on $J_0$ are the same as $J_0 \subsetneq J''_\pm$ 
for ${\tt P_1}$, and $J_0 \supsetneq J''$ for ${\tt P_2}$.

\begin{equation*}
\left({\tt P_3}, h_{I_0}\otimes h_{J_0}\right) = \left\{ \begin{array}{rl}
	- \left([T_1, b_{J''}]p, h_{I_0}\right) \frac{1}{\sqrt{|J''|}},  & \text{ if } 
			J_0 = J''_- \text{ or }	J_0 = J''_+\\
	0, & \text{ otherwise.}
\end{array}\right.
\end{equation*}

\begin{equation*}
\left({\tt P_4}, h_{I_0}\otimes h_{J_0}\right) = \left\{ \begin{array}{rl}
	- \left([T_1, b_{J_0}]p, h_{I_0}\right) h_{J''_+}(J_0),  & \text{ if } J_0 \subsetneq J''_+	\\
	+ \left([T_1, b_{J_0}]p, h_{I_0}\right) h_{J''_-}(J_0),  & \text{ if } J_0 \subsetneq J''_-	\\
	0, & \text{ otherwise.}
\end{array}\right.
\end{equation*}

\begin{equation*}
\left({\tt P_5}, h_{I_0}\otimes h_{J_0}\right) = \left\{ \begin{array}{rl}
	- \left([T_1, b_{J''_+}]p, h_{I_0}\right) h_{J_0}(J''_+)
	+ \left([T_1, b_{J''_-}]p, h_{I_0}\right) h_{J_0}(J''_-), & \text{ if }	J_0 \supseteq J''\\
	0, & \text{ otherwise.}
\end{array}\right.
\end{equation*}

Note that for ${\tt P_1 + P_2 + P_3}$, the second component Haar functions run through all the \textit{odd} dyadic intervals intersecting $J''$. For ${\tt P_4 + P_5}$, the second component Haar functions run through \textit{all} dyadic intervals intersecting $J''$, except $J''_{\pm}$. Pick the components of ${\tt P_4}$ and ${\tt P_5}$ with even $J$ on the second coordinate:
	\begin{eqnarray*}
	{\tt P}_{\text{even}} &:=& - \sum_{\substack{J \subsetneq J''_+ \\ J \text{ even}}} 
					\left([T_1, b_J]p\right) h_{J''_+}(J) h_J +
					\sum_{\substack{J \subsetneq J''_- \\ J \text{ even}}} 
					\left([T_1, b_J]p\right) h_{J''_-}(J) h_J\\
				&& - \left([T_1, b_{J''_+}]p\right) \sum_{\substack{J \supsetneq J''_+ \\ J \text{ even}}}
					h_J(J''_+)h_J
				+ \left([T_1, b_{J''_-}]p\right) \sum_{\substack{J \supsetneq J''_- \\ J \text{ even}}}
					h_J(J''_-)h_J.
	\end{eqnarray*}
Then $\cT^b(p\otimes\hdp) - {\tt P}_{\text{even}}$ and ${\tt P}_{\text{even}}$ are \textit{orthogonal}, so if $\cT^b$ is bounded then these components must be individually bounded. In particular, ${\tt P}_{\text{even}}$ is bounded.
Clearly
	\begin{eqnarray*}
	\|{\tt P}_{\text{even}}\|_2^2 &\geq& \sum_{\substack{J \subsetneq J''_+ \\ J \text{ even}}}
		\sum_{I_0} \left([T_1, b_J]p, h_{I_0}\right)^2 \frac{1}{|J''_+|}
		+ \sum_{\substack{J \subsetneq J''_- \\ J \text{ even}}}
		\sum_{I_0} \left([T_1, b_J]p, h_{I_0}\right)^2 \frac{1}{|J''_-|}\\
	&=& \frac{2}{|J''|} \sum_{\substack{J \subsetneq J'' \\ J \text{ even}}} \left\| [T_1, b_J]p \right\|_2^2,
	\end{eqnarray*}
giving us that
	\begin{equation}
	\label{E:rec-evens}
	\sum_{\substack{J \subsetneq J'' \\ J \text{ even}}} \left\| [T_1, b_J]p \right\|_2^2 
		\lesssim |J''| \|p\|_2^2, \:\text{ for all } J'' \text{ even}.
	\end{equation}
	
Now suppose $J''$ is \textit{odd}, so we only have the operators ${\tt P_1}$ and ${\tt P_2}$. As before,
\begin{equation*}
\left({\tt P_1}, h_{I_0}\otimes h_{J_0}\right) = \left\{ \begin{array}{rl}
	\left([T_1, b_{J_0}]p, h_{I_0}\right) s(J_0, \hat J_0) \hdp (\hat J_0), & \text{ if } J_0 \text{ is odd, and }
		J_0 \subsetneq J''\\
	0, & \text{ otherwise.}
\end{array}\right.
\end{equation*}

\begin{equation*}
\left({\tt P_2}, h_{I_0}\otimes h_{J_0}\right) = \left\{ \begin{array}{rl}
	\left([T_1, b_{J''}]p, h_{I_0}\right) s(J_0, \hat J_0) h_{\hat J_0}(J''),  & \text{ if } J_0
		\text{ is odd, and } J_0\supsetneq J''	\\
	0, & \text{ otherwise.}
\end{array}\right.
\end{equation*}

Obviously ${\tt P_1}$ and ${\tt P_2}$ are orthogonal, so
	$\|\cT^b(p\otimes\hdp)\|_2^2 = \|{\tt P_1}\|^2 + \|{\tt P_2}\|_2^2$ for odd  $J''$.
From $\|{\tt P_1}\|_2^2 \leq \|\cT^b\|_2^2$, it is easy to see that
	\begin{equation}
	\label{E:rec-odds}
	\sum_{\substack{J \subsetneq J'' \\ J \text{ odd}}} \left\| [T_1, b_J]p \right\|_2^2 
		\lesssim |J''| \|p\|_2^2, \:\text{ for all } J'' \text{ odd}.
	\end{equation}

From \eqref{E:rec-evens} and \eqref{E:rec-odds}, it is obvious that, if $\cT^b$ is bounded, then
	\begin{equation}
	\label{E:rec-sum}
	\sum_{J \subset J''} \left\| [T_1, b_J] \right\|_2^2 \lesssim |J''|, \text{ for all } J''\in\cD.
	\end{equation}
This is a much stronger assertion than the Lemma needs, and immediately implies \eqref{E:1parBMOx} using the one-parameter implication: 
$$\frac{1}{\sqrt{|J|}}[T_1, b_{J}] \text{ is } L^2\text{-bounded} \Rightarrow \frac{1}{\sqrt{|J|}}b_{J} \in BMO(\bR)
\Rightarrow \sum_{I' \subset I} \left|(b, h_{I'}\otimes h_J)\right|^2 \lesssim |I||J|.
$$
\end{proof}

\begin{theorem}
\label{T:BMO-rec}
If $\cT^b$ is bounded, then $b \in BMO_r$.
\end{theorem}

\begin{proof}
The results in Lemma \ref{L:TbonP} immediately show that all operators in Groups I and II are a priori bounded and can be disregarded. What remains are the four operators in Group III.

\textbf{Case 1: $I''$ and $J''$ are both odd.} In this case, we only have the operator $\pi\pi1$, which we rewrite as 
	$$\pi\pi1 = - \sum_{\substack{I \text{ odd}, I\subsetneq I'' \\ J \text{ odd}, J \subsetneq J''}} 
	(b, h_{\hat I}\otimes h_{\hat J}) h_{I''}(\hat I) h_{J''}(\hat J) s(I, \hat I) s(J, \hat J) h_I \otimes h_J.$$
So boundedness of $\cT^b$ in this case reduces to just this $\pi\pi1$. Computing its norm gives us
	$$\|\pi\pi1\|_2^2 = \frac{1}{|I''||J''|}\sum_{\substack{I \text{ odd}, I\subsetneq I'' \\ J \text{ odd}, J \subsetneq J''}} 
		|(b, h_{\hat I}\otimes h_{\hat J})|^2 
		= \frac{1}{|I''||J''|} \sum_{\substack{I \text{ even}, I\subsetneq I'' \\ J \text{ even}, J \subsetneq J''}} 
		|(b, h_{I}\otimes h_{J})|^2 
		\lesssim C,$$
and from here immediately
	\begin{equation}
	\label{E:BMO-rec1}
	\sum_{\substack{I \text{ even}, I\subseteq I_0 \\ J \text{ even}, J \subseteq J_0}} 
		|(b, h_I\otimes h_J)|^2 \lesssim |I_0||J_0|, \text{ for all } I_0, J_0 \in \cD.
	\end{equation}

\vspace{0.2in}
\textbf{Case 2: $I''$ is odd and $J''$ is even.} In this case we have two operators, $\pi\pi1$ and $\pi\pi2$. But now we know from \eqref{E:BMO-rec1} that $\pi\pi1$ is already bounded, so we only need to look at $\pi\pi2$:
	$$\pi\pi2 = - \sum_{\substack{I \text{ odd}, I\subsetneq I'' \\ J \subsetneq J''_+}} (b, h_{\hat I}\otimes h_J) 
		h_{I''}(\hat I) h_{J''_+}(J) s(I, \hat I) h_I \otimes h_J
+ \sum_{\substack{I \text{ odd}, I\subsetneq I'' \\ J \subsetneq J''_-}} (b, h_{\hat I}\otimes h_J) 
		h_{I''}(\hat I) h_{J''_-}(J) s(I, \hat I) h_I \otimes h_J.$$
		
Separate $\pi\pi2$ into $\pi\pi2_{\text{odd}}$, where we only sum over \textit{odd} $J \subsetneq J''_\pm$, and
$\pi\pi2_{\text{even}}$, where we only sum over \textit{even} $J \subsetneq J''_\pm$. These are orthogonal, so they are individually bounded. Now
	$$\|\pi\pi2_{\text{odd}}\|_2^2 = \frac{2}{|I''||J''|} 
	\sum_{\substack{I \text{ even}, I\subsetneq I'' \\ J \text{ odd}, J\subsetneq J''_\pm}} 
	|(b, h_I \otimes h_J)|^2 \lesssim C.$$
From here, it follows easily that
	\begin{equation}
	\label{E:BMO-rec2}
	\sum_{\substack{I \text{ even}, I\subseteq I_0 \\ J \text{ odd}, J \subseteq J_0}} 
		|(b, h_I\otimes h_J)|^2 \lesssim |I_0||J_0|, \text{ for all } I_0, J_0 \in \cD.
	\end{equation}

\vspace{0.2in}
\textbf{Case 3: $I''$ is even and $J''$ is odd.} This case is symmetrical to Case 2, and yields
	\begin{equation}
	\label{E:BMO-rec3}
	\sum_{\substack{I \text{ odd}, I\subseteq I_0 \\ J \text{ even}, J \subseteq J_0}} 
		|(b, h_I\otimes h_J)|^2 \lesssim |I_0||J_0|, \text{ for all } I_0, J_0 \in \cD.
	\end{equation}

\vspace{0.2in}
\textbf{Case 4: $I''$ and $J''$ are both even.} In this case we have all four operators in Group III, but we already know from \eqref{E:BMO-rec1}, \eqref{E:BMO-rec2} and \eqref{E:BMO-rec3} that all but $\pi\pi4$ are already bounded. Now
	$$\|\pi\pi4\|_2^2 = \frac{4}{|I''||J''|} \sum_{\substack{I \subsetneq I''_\pm \\
		J \subsetneq J''_\pm}} |(b, h_I \otimes h_J)|^2,$$
which gives us exactly that
	$$\sum_{\substack{I \subseteq I_0 \\ J \subseteq J_0}} |(b, h_I \otimes h_J)|^2 \lesssim |I_0||J_0|,
		\text{ for all } I_0, J_0 \in \cD,$$
so $b \in BMO_r$.
\end{proof}


\section{Analyzing what is left to estimate if $BMO_r$ of the symbol is small}
\label{small}

In the previous section we proved
\begin{equation}
\label{r}
\|b\|_{BMO_r} \lesssim \|\cT^b\|\,.
\end{equation}
We can normalize $b$ to have 
$$
\|b\|_{BMO_{ChF}}=1\,.
$$
Hence, to get an estimate of $\cT^b$ from below we can assume
\begin{equation}
\label{r}
\|b\|_{BMO_r} \lesssim \eps_0\,,
\end{equation}
where $\eps_0$ is a very small positive absolute number.

\medskip

 We recall that everything is dyadic.

\begin{defin}
The Haar  support of  function $\al$ is the collection of dyadic rectangles such that Haar coefficient does not vanish:
$\sigma_H(\al):=\{R: (\al, h_R)\neq 0\}$.
\end{defin}

\begin{defin}
Let $R= I\times J$, $p(R):= \hat R=\hat I\times \hat J$, $p_m(R)=\hat p_{m-1}(R)$.
\end{defin}

\begin{defin}
Let $U_0$ be a finite union of dyadic rectangles. We call such $U_0$ {\it a dyadic open set}. For every such $U_0$ we denote by $\cU_0$ the collection of all dyadic rectangles lying inside $U_0$.
\end{defin}

\begin{defin}
We say that function $\al$ is deep inside a dyadic open set $U$ if $p_m (\sigma_H(\al)) \subset \cU$. We can call such $\al$ $m$-deep inside $U$.
\end{defin}

\begin{defin}
$\cT^b_* = [T_2^*,[T_1^*, b]]$.
\end{defin}

\subsection{Plan}
We will try to estimate $\cT^b$ from below under assumptions:
\begin{enumerate}
\item $\sigma_H(b)$ is finite;
\item $\|b\|_{BMO_{ChF}} =1$;
\item $ \|b\|_{BMO_r} \le \eps_0$;
\item $U_0$ is a dyadic open set such that $\|b\|_{BMO_{ChF}} = \frac1{|U_0|} \sum_{R\in \cU_0} (b, h_R)^2$.
\end{enumerate}
By rescaling we can assume that $|U_0|=1$.

We will stay soon only with $\pi\pi$ part of repeated commutator.

\begin{lemma}( Symbol is ``smaller'' than test function)
\label{smsy}
 $\cT^{h_{I'}\otimes h_{J'}}_* (h_{I''}\otimes h_{J''}) \neq 0 \, \Rightarrow I' \le \hat{I''}, J' \le \hat{J''}$.
 \end{lemma}
 \begin{proof}
 Direct calculation.
 \end{proof}
 
 \begin{lemma}( Symbol is ``smaller'' than test function)
\label{smsy1}
 $\cT^{h_{I'}\otimes h_{J'}} (h_{I''}\otimes h_{J''}) \neq 0 \, \Rightarrow I' \le {I''}, J' \le {J''}$.
 \end{lemma}
 \begin{proof}
 Direct calculation.
 \end{proof}

 \begin{defin} 
 \label{UU0}
 Given $U_0$, we use a standard notation $U$ for the set $=\{M_s^d \one_{U_0} \ge \frac1{16}\}$, where $M_s^d$ denotes the strong dyadic maximal function.
 \end{defin}
 
 \begin{lemma}
 \label{1deep}
Let $U_0$ be a finite union of dyadic rectangles. Let $U$ be as in the Definition \ref{UU0}. Let $\rho$ be such that $\sigma_H(\rho) \subset \cU_0$. Then $\cT^\beta_*(\rho)  =0$  for any $\beta$ such that $\sigma_H(\beta)\cap \cU =\emptyset$.
\end{lemma}
\begin{proof}
 Follows from Lemma \ref{smsy}.
 \end{proof}
 
  \begin{lemma}
 \label{1deepstar}
Let $U_0$ be a finite union of dyadic rectangles. Let $U$ be as in the Definition \ref{UU0}. Let $\rho$ be such that $\sigma_H(\rho) \subset \cU$. Then $\cT^\beta(\rho)  =0$  for any $\beta$ such that $\sigma_H(\beta)\cap \cU =\emptyset$.
\end{lemma}
\begin{proof}
 Follows from Lemma \ref{smsy1}.
 \end{proof}
 
 \bigskip
 
 \begin{defin}
$Edge_m(U_0) =\{R: R\in \cU_0, p_{m-1}(R) \in \cU_0, p_m(R) \cap U_0^c \neq \emptyset\}$. $Edge(U_0) := Edge_1(U_0)$. The same for $U$ from Definition \ref{UU0}. $Edge (U) =\{R: R\subset U, \hat R \cap U^c \neq \emptyset\}$.
\end{defin}

So $Edge_m(U_0)$ describes rectangles in $U_0$ that  are ``deep, but not so deep'' in $U_0$, they are on the edge.

\bigskip

Given $b$ and $U_0$, we consider $U$ built as in Definition \ref{UU0} and we consider the decomposition
\begin{equation}
\label{deco}
b=\al_0 + \beta,
\end{equation}
where 
\begin{enumerate}
\item $\al_0:= \sum_{R\,\subset  U} (b, h_R) h_R$;
\item $\beta$ is the rest, meaning that rectangles giving non-zero $\beta$ coefficient have to be like this: $R\cap U^c\neq\emptyset$.
\end{enumerate}

By Lemma \ref{1deep} that says that symbol should be ``smaller" than the test function,  we get that if  $\rho$ is such that $\sigma_H(\rho) \subset \cU_0$, then
\begin{equation}
\label{gamma-est}
\cT^\beta_*(\rho)=0\,.
\end{equation}

\bigskip

So--in estimating $\cT^b_*(\rho)$ from below we can forget about $\|\cT^\beta_*(\rho)\|_2$ and look only at $\cT^{\al_0}_*(\rho)$.  And the only thing we used for claiming that is that Haar spectrum of $\rho$ is in $\cU_0$ and Definition \ref{UU0} of $U$.

By Definition \ref{UU0} and weak boundedness of $M_s^d$ we know that
\begin{equation}
\label{weakM}
|U|\le C_0 |U_0|\,.
\end{equation}

Obviously, if we normalize $|U_0|=1$, $\|\rho\|_2=1$, we have
$$
C_0^{1/2}\|\cT^b\| \ge \|\cT^b\|  |U|^{1/2}  \ge  \|\cT^b\| \|\rho\| |U|^{1/2} \ge  |( \one_U, \cT^b_*(\rho))|= |( \one_U, \cT^{\al_0}_*(\rho))|\,.
$$

\bigskip

So we are left to estimate $ |( \one_U, \cT^{\al_0}_*(\rho))|$ from below, where $\sigma_H(\rho)\subset \cU_0$.

\section{The estimate  of $ |( \cT^{ \al_0}(\one_U), \rho|$ from below,  where $\sigma_H(\rho)\subset \cU_0$}

To estimate $ |( \one_U, \cT^{\al_0}_*(\rho))|$ from below let us look at our formulas for $\pi\pi, \pi Z, \dots, DD$ for
$$
b= \al_0,\,\, f =\one_U\,.
$$
\begin{remark}
\label{m}
If needed we can always replace $\al_0$ that was such that $\sigma_H(\al_0)\subset \cU$, by $\al_{m-1}$, that is
$\sigma_H(\al_{m-1})\subset \cU$. This is only if we need this. The result one obtains after such change is weaker, but it still would give the information on some part of Chang--Fefferman $BMO$ norm of $b$.
\end{remark}

\subsection{Signature and parts of commutator that vanish}
We use the notion of signature. We look at the functions employed in the definition of each part. There are $3$ functions in $I$ position and $3$ functions in $J$ position. We write signature as a matrix, where the first row is assigned to $3$ functions in $I$ position and the second row is assigned to $3$ functions in $J$ position. The element of the matrix is $0$ if the integral of corresponding function is $0$ (e.g. $h_I$), and  the element of the matrix is $1$ if the integral of corresponding function is not $0$ (e.g. $\one_I$). For example,
$$
\begin{bmatrix} 0,0,0
\\
0,0, 0\end{bmatrix}
$$
is the signature of $DD$ part of repeated commutator, and
$$
\begin{bmatrix} 0,0,0
\\
0,1, 0\end{bmatrix}\,.
$$
is the signature of $D\pi$ part. $DZ$ part has signature
$$
\begin{bmatrix} 0,0,0
\\
0,0, 1\end{bmatrix}\,.
$$

\begin{lemma}
\label{11}
All parts having only one $1$ part in the signature can be estimates by $C\|\al_0\|_{BMO_r}\le C\eps_0$.
\end{lemma}
\begin{proof}
Direct calculation.
\end{proof}

Thus those parts can be disregarded. These are all $D...$, $...D$ parts.

Consider   now parts with $Z...$, $...Z$.  These are terms with signatures
$$
\begin{bmatrix} 0,1,0
\\
0,0, 1\end{bmatrix}\,,
$$
and
$$
\begin{bmatrix} 0,0,1
\\
0,1, 0\end{bmatrix}\,.
$$
Also $ZZ$ has signature
$$
\begin{bmatrix} 0,0,1
\\
0,0, 1\end{bmatrix}\,.
$$

They will necessarily consist of the  terms having the  following factors:
\begin{enumerate}
\item $ (\one_U ,  h_I\otimes h_J)$;
\item $ (\one_U ,  h_I\otimes \tilde \one_J)$;
\item $ (\one_U , T_1^* h_I\otimes \tilde \one_J)$;
\item $ (\one_U ,  \tilde \one_I\otimes T_1^* h_J)$;
\item $ (\one_U ,  h_I\otimes T_2^*\tilde \one_J)$;
\item $ (\one_U , T_1^* h_I\otimes T_2^*h_J)$;
\item et cetera... .
\end{enumerate}
It is important to remember that here we can count only those $I, J$ such that $ I \times  J$ are inside $U_0$.  This follows from our choice of $\rho$: $\sigma_H(\rho)\subset \cU_0$.
But as $ I \times  J$ are inside $U_0$ almost all those scalar products above are equal to zero. Unfortunately not all, number (5) does not vanish as $T_2^*\tilde \one_J$ is a very non-localized function. But let us even assume for a while that all such terms disappear (which is false as we  just explained, some terms stay).

\bigskip

Let $\rho$ be an arbitrary test function with Haar spectrum inside $\cU_0$ (we can actually use functions with Haar spectrum  $m$-deep in $U_0$ if needed).
Build $U$ as in Definition \ref{UU0}. We have \eqref{weakM}. We denoted by $\al_0$ the part of $b$ having Haar spectrum in $\cU$.

After the bookkeeping we see that in $|(\one_U, \cT^{\al_0}_* \rho)|$ only $\pi\pi$ part rests because all other parts either vanish or are smaller than $C\|b\|_{BMO_{r}} \le C\eps_0$.

Let  us write down all four terms of $\pi\pi$ part of $(\one_U, \cT^{\al_0}_*\rho)$, $\sigma_H(\rho)\subset U_0$, 
\begin{align}
\label{ppagain}
\pi\pi(\al_0, \one_U; \bfT)= &\sum (\al_0, h_I\otimes h_J) (\one_U , \tilde\one_I\otimes \tilde\one_J) T_1 h_I\otimes T_2 h_J -
\\
&\sum (\al_0, h_I\otimes h_J) (\one_U , T_1^*\tilde\one_I\otimes \tilde\one_J)  h_I\otimes T_2 h_J -\notag
\\
&\sum (\al_0, h_I\otimes h_J) (\one_U , \tilde\one_I\otimes T_2^*\tilde\one_J) T_1 h_I\otimes  h_J +\notag
\\
&\sum (\al_0, h_I\otimes h_J) (\one_U , T_1^*\tilde\one_I\otimes T_2^*\tilde\one_J)  h_I\otimes h_J= I -II-III+IV\,.\notag
\end{align}

The range of $I, J$ can be naturally associated with the definition of $T_1, T_2$, and with the nature of $\rho$: $I\times J\subset U_0$.

\bigskip

\subsection{Conjugate operator}
It seems to be obvious that we can replace both $T_i$  by $T_i^*$, and get the similar expression.
Now we want to estimate $ |( \one_U, \cT^{\al_0}_*(\rho))|$ from below. Let $\rho$ be such that $\sigma_H(\rho)\subset \cU$. We denoted by $\al_0$ the part of $b$ having Haar spectrum in $\cU$. By Lemma \ref{smsy1}
$$
( \one_U, \cT^{\al_0}(\rho)) =( \one_U, \cT^{b}(\rho))\,.
$$

In estimating $( \one_U, \cT^{\al_0}(\rho))$ we  the  terms having the  following factors:
\begin{enumerate}
\item $ (\one_U , h_I\otimes h_J)$;
\item $ (\one_U , h_I\otimes \tilde \one_J)$;
\item $ (\one_U , T_1 h_I\otimes \tilde \one_J)$;
\item $ (\one_U ,  \tilde \one_I\otimes T_1 h_J)$;
\item $ (\one_U ,  h_I\otimes T_2\tilde \one_J)$;
\item $ (\one_U , T_1 h_I\otimes T_2h_J)$;
\item et cetera... .
\end{enumerate}
By the choice of $\al_0$ just mentioned above, $I\times J\subset U$ here. Therefore all these scalar products seem to vanish. But the attentive reader can notice that, yes, e.g. $(\one_U , T_1 h_I\otimes \tilde \one_J)$ does vanish, but $ (\one_U ,  h_I\otimes T_2\tilde \one_J)$ does not. This will create a big difficulty below.

\bigskip

However, let us assume for a while that
 in estimating $( \one_U, \cT^{\al_0}(\rho))$ from below we will be left with sum of terms estimated by $C\|b\|_{BMO_{r}}$ and with the following $\pi\pi$ terms that we cannot cancel out:
\begin{align}
\label{ppagainstar}
\pi\pi(\al_0, \one_U; \bfT^*)= &\sum (\al_0, h_I\otimes h_J) (\one_U , \tilde\one_I\otimes \tilde\one_J) T_1^* h_I\otimes T_2^* h_J -
\\
&\sum (\al_0, h_I\otimes h_J) (\one_U , T_1\tilde\one_I\otimes \tilde\one_J)  h_I\otimes T_2^* h_J -\notag
\\
&\sum (\al_0, h_I\otimes h_J) (\one_U , \tilde\one_I\otimes T_2\tilde\one_J) T_1^* h_I\otimes  h_J +\notag
\\
&\sum (\al_0, h_I\otimes h_J) (\one_U , T_1\tilde\one_I\otimes T_2\tilde\one_J)  h_I\otimes h_J= A -B-C+D\,.\notag
\end{align}
We do not write the scalar product with $\rho, \sigma_H(\rho) \subset \cU$, but we keep it in mind.
The range of $I, J$ can be naturally associated with the definition of $T_1, T_2$, and with the nature of $\al_0$: definitely $I\times J\subset U$, but for example the term with $ T_1^* h_I\otimes T_2^* h_J$ disappears if $I\times J\subset U$, but $ \hat I\times \hat J $ is not inside $U$. This term will be ``cut-out''  by the scalar product with $\rho, \sigma_H(\rho) \subset \cU$.

But notice that all terms such that $I\times J\subset U_0$ stay, they are not cut-out. This is because  by  Definition \ref{UU0} we have
$ \hat I\times \hat J \subset U$ as soon as $I\times J\subset U_0$.

\bigskip

We like the first line of \eqref{ppagainstar} (and we also like the first line of \eqref{ppagain}). If the first line were alone in \eqref{ppagainstar} then:
$$
\|\pi\pi(\al_0, \one_U; \bfT^*)\|_2^2=\|A\|_2^2 \ge \sum (\al_0, h_R)^2 = \sum_{\stackon{R\, \subset \,U_0}{
R \,odd\, odd}} (b, h_R)^2\,.
$$
And we would get
\begin{equation}
\label{est-ee}
\|\cT^b\| =\|\cT^b_*\| \ge c\,\sum_{\stackon{R \,\subset \,U_0}{
R \, odd\, odd}} (b, h_R)^2\,,
\end{equation}
which is ``almost'' what we want (in fact that would be $25$ percent of what we want). Of course we want more, namely,  the summation not restricted to  just even-even $R$. So we have several problems:
\begin{enumerate}
\item how to get the summation over all $R$'s inside $U_0$?
\item what to do with terms $II$, $III$, $IV$?
\end{enumerate}

Of course we need summation over all $R$ not just (odd, odd), but still, \eqref{est-ee} would be a step in the direction of proving the ultimate estimate:
\begin{equation}
\label{est-ult}
\|\cT^b\| \ge c\,\sum_{R \,\subset \,U_0} (b, h_R)^2\,.
\end{equation}
\bigskip

\subsection{Coming back to $(\one_U, \cT^{\al_0}_*(\rho))$ and trying other test functions}
\begin{remark}
\label{TU}
We can try to use $T_1 \one_U, T_2 \one_U, T_1T_2 \one_U$ as test functions instead of $\one_U$ in \eqref{ppagain}. Again it is easy to see that if we choose  instead of $\al_0$ a deeper function $\al_1= \sum_{R\,is\, 2-deep\, in\, U_0} (b, h_R) h_R$, then after bookkeeping we see that in $|(T_i\one_U, \cT^{\al_1}(\al_1))|$ only the $\pi\pi$ part rests because all other parts either vanish or are smaller than $C\|b\|_{BMO_{r}}\le C\eps_0$.
\end{remark}
Let us look at (we use that $(T_1^*)^2=0$)
\begin{align}
\label{ppT1}
\pi\pi(\al_1, T_1\one_U)= &\sum (\al_1, h_I\otimes h_J) (\one_U , T_1^*\tilde\one_I\otimes \tilde\one_J) T_1 h_I\otimes T_2 h_J +
\\
&\sum (\al_1, h_I\otimes h_J) (\one_U ,T_1^* \tilde\one_I\otimes T_2^*\tilde\one_J) T_1 h_I\otimes  h_J \,.
\notag
\end{align}
We immediately conclude that the second sum restricted to $I$ and $J$ even is orthogonal to what is rest, so
\begin{equation}
\label{ee}
\sum_{I, J \, even, even} (\al_1, h_I\otimes h_J)^2 (\one_U ,T_1^* \tilde\one_I\otimes T_2^*\tilde\one_J)^2 \le C\|\cT^b\|^2\,.
\end{equation}
Let us look at (we use that $(T_2^*)^2=0$)
\begin{align}
\label{ppT1}
\pi\pi(\al_1, T_2\one_U)= &\sum (\al_1, h_I\otimes h_J) (\one_U , \tilde\one_I\otimes T_2^*\tilde\one_J) T_1 h_I\otimes T_2 h_J +
\\
&\sum (\al_1, h_I\otimes h_J) (\one_U ,T_1^* \tilde\one_I\otimes T_2^*\tilde\one_J)  h_I\otimes  T_2h_J \,.
\notag
\end{align}
We again conclude that the second sum restricted to $I$ and $J$ even is orthogonal to what is rest, so the same inequality follows:
\begin{equation}
\label{eee}
\sum_{I, J \, even, even} (\al_1, h_I\otimes h_J)^2 (\one_U ,T_1^* \tilde\one_I\otimes T_2^*\tilde\one_J)^2 \le C\|\cT^b\|^2\,.
\end{equation}

Exactly the same inequality follows if we  consider $\pi\pi(\al_1, T_1T_2\one_U)$.

But what would be needed is the ``opposite" sum estimate:
\begin{equation}
\label{oo}
\sum_{I, J \, odd, odd} (\al_1, h_I\otimes h_J)^2 (\one_U ,T_1^* \tilde\one_I\otimes T_2^*\tilde\one_J)^2 \le C\|\cT^b\|^2\,,
\end{equation}
because exactly this sum entangles with the first line of  \eqref{ppagain}: $\sum (\al_0, h_I\otimes h_J) (\one_U , \tilde\one_I\otimes \tilde\one_J) T_1 h_I\otimes T_2 h_J $.

\bigskip

{\bf  Problem.} So far we did not guess how to get \eqref{est-ee} from  estimating  formula \eqref{ppagainstar}, not to mention that \eqref{est-ult} rests even less clear. But we will proceed in the next two sections.

\section{One parameter case and \eqref{ppagainstar}}
\label{back1p}

Let us take a look at the one parameter case. We obviously can assume that we have localization, and that eventually we just need to find out how to estimate the analog of, say, \eqref{ppagainstar}  from below.

Here is this analog. It has two rather than four sums because we are in one parameter situation. The ``open'' set $U$ in one parameter case is just a dyadic interval $I_0$. We can assume that
$$
|I_0|=1\,.
$$
\begin{align}
\label{pipi1p}
\pi(\al_0, \one_{I_0}; T^*)=&\sum_{I\in \cD(I_0)} (b, h_I) \La T^* \one_{I_0}\Ra_I h_I -\sum_{I\in \cD(I_0)} (b, h_I) \La \one_{I_0}\Ra_I T^* h_I=
\\
&\sum_{I\in \cD(I_0)} (b, h_I) \La T^* \one_{I_0}\Ra_I h_I -\sum_{I\in \cD(I_0)} (b, h_I)  T^* h_I=\notag
\\
&A-B\,.\notag
\end{align}

Before plunging into the consideration of \eqref{pipi1p}, let us notice that the choice of $\one_U$ in two parameter case was initiated by the wish to stay only with $\pi\pi$ part and to bring many other parts of the commutator to zero.

Such a problem is not present in one parameter case. In this case we can notice that the estimate of the commutator can be reduced to the estimate from below of the following expression, where the test function $f\neq \one_{I_0}$.
In fact we can choose test function differently.
So consider the following more general expression
\begin{align}
\label{pipi1pf}
\pi(\al_0, f_{I_0}; T)=&\sum_{I\in \cD(I_0)} (b, h_I) \La T^* f_{I_0}\Ra_I h_I -\sum_{I\in \cD(I_0)} (b, h_I) \La f_{I_0}\Ra_I T^* h_I=
\\
&A_f-B_f\,.\notag
\end{align}

Let us now choose a very simple $f_{I_0}\neq \one_{I_0}$. Here it is convenient to think that $I_0$ is odd (and normalized $|I_0|=1/2$). In this case our choice of test function is the ``uncle" function:
$$
f_{I_0}= h_{sI_0}\,.
$$
Here $sI_0$ is a sibling of $I_0$.  By the definition of $T^*$,
$$
h_{sI_0}|I_0 =0,\quad T^*h_{sI_0}|I_0 = \sign(I_0, \hat I_0) h_{\hat I_0} = \pm \frac1{\sqrt{2}}\,.
$$
We plug into \eqref{pipi1pf} above and $B_f$ disappears.  The estimate of $A_f$ gives
$$
\sum_{I\in \cD(I_0)} (b, h_I)^2 \le C\|\pi(\al_0, f_{I_0}; T)\|_2^2\le C\|\text{the commutator}\|_2^2\,.
$$
In fact, for even $I_0$ the choice $f= h_{s(\hat I_0)}$ would work.
We are done.

\bigskip

\begin{remark}
\label{whatU}
What serves as ``the uncle''  function $F= f_{U}$ of $U$ in the two parameter case--this is not clear.
\end{remark}

\bigskip

Now let us come back to \eqref{pipi1p}, and let us try to obtain 
$
\sum_{I\in \cD(I_0)} (b, h_I)^2 \le C\|\text{the commutator}\|_2^2
$
directly from  \eqref{pipi1p}.
 
\bigskip

Recall that $T^* h_I= \sign(I, \hat I) h_{\hat I}$ for odd $I$ and $0$ for even.

Let us consider $I\in \cD(I_0)$ such that $\hat I\in \cD(I_0)$, and then we have a term $-\sign(I, \hat I)(b, h_I) h_{\hat I}$ in $-B$ sum (this is when $I$ is odd, when $I$ is even this term vanishes), and we have another term
$h_{\hat I} \cdot (b, h_{\hat I})$ with coefficient $c_I:= \La T^* \one_{I_0}\Ra_{\hat I}$. This term is always present, for even and odd $I$.

\bigskip

Let us understand, what is $c_I$, which is equal to the constant value of the constant function $T^* \one_{I_0}| I_0$ (recall that $\hat I\subset I_0$). Let $I_1$ denote the father of $I_0$, $I_2$ denote the father of $I_1$, et cetera.
$$
\one_{I_0}= (\one_{I_0}, h_{I_1}) h_{I_1} + (\one_{I_0}, h_{I_2}) h_{I_2}+\dots\,.
$$
Hence
$$
(T^*\one_{I_0})|I_0= (\one_{I_0}, h_{I_1}) (T^*h_{I_1})|I_0 + (\one_{I_0}, h_{I_2}) (T^*h_{I_2})|I_0+\dots\,.
$$
We recall that $|I_0|=1$. Definitely the first term of the series is at most $\frac1{2\sqrt{2}}$ by absolute value (it can be zero too). In fact, $ h_{I_1}$ gives the contribution of $\frac1{\sqrt{2}}$ and $T^*h_{I_1}$ gives the contribution $\frac12$. The next term of the series is at most $\frac1{2\sqrt{2}}$: $ h_{I_2}$ gives the contribution of $\frac12$ and $T^*h_{I_2}$ gives the contribution $0$ or  $\frac1{2\sqrt{2}}$.
It is clear that
\begin{equation}
\label{cI}
\frac18\le |c_I| \le \frac1{\sqrt{2}} <1\,.
\end{equation}

We see that  for $I$ such that $\hat I\subset I_0$ the sum $A-B$ of \eqref{pipi1p} has the term
$$
h_{\hat I}\cdot (-\sign(I, \hat I)\cdot [(b, h_I) - d_I (b, h_{\hat I})], \quad |d_I| \le \frac1{\sqrt{2}}\,,
$$
when $I$ is odd, and the  term
$$
h_{\hat I}\cdot (\sign(I, \hat I)\cdot d_I (b, h_{\hat I}), \quad |d_I| \ge \frac18\,,
$$
when $I$ is even.

\bigskip

All terms are orthogonal, and we get
\begin{equation}
\label{mi}
\sum_{I: \hat I\in \cD(I_0), I odd} |(b, h_I) - d_I (b, h_{\hat I})|^2 \le C\|\text{the commutator}\|_2^2\,,
\end{equation}
and
\begin{equation}
\label{miO}
\sum_{I: I\in \cD(I_0), I odd} |(b, h_{I})|^2 \le \sum_{I: \hat I\in \cD(I_0), I even} |(b, h_{\hat I})|^2 \le C\|\text{the commutator}\|_2^2\,,
\end{equation}
From these two inequalities and from the estimates of $d_I$, we can immediately obtain
\begin{equation}
\label{miA}
\sum_{I: I\in \cD(I_0)} |(b, h_{I})|^2 \le  C\|\text{the commutator}\|_2^2\,,
\end{equation}

\bigskip

Notice that \eqref{mi} alone would give us an interesting information. 

We want to get from \eqref{mi} alone  that
\begin{equation}
\label{1deep}
\sum_{I: \hat I\in \cD(I_0), I odd} (b, h_I)^2 \le C\|\pi(\al_0, f_{I_0}; T)\|_2^2\le C\|\text{the commutator}\|_2^2\,.
\end{equation}
(This is stated in \eqref{miO}, but suppose that we do not have it,  suppose that we have only \eqref{mi}.)
If this is obtained, then it is easy to get
\begin{equation}
\label{vse}
\sum_{I\in \cD(I_0), I odd} (b, h_I)^2 \le C\|\pi(\al_0, f_{I_0}; T)\|_2^2\le C\|\text{the commutator}\|_2^2\,.
\end{equation}

Now \eqref{mi} $\Rightarrow$ \eqref{1deep} seems is easy by the following lemmas.
\begin{lemma}
\label{Schur}
Let $|d_n|\le q<1$, let $b_n: = a_n - d_n a_{n-1}$. Assume that $b:=\{b_n\}_{n=0}^\infty \in \ell^2$, then $a:=\{a_n\}_{n=0}^\infty \in \ell^2$.
\end{lemma}
\begin{proof}
We can write
$$
a_n = b_n + d_n a_{n-1} = b_n + d_n b_{n-1} + d_nd_{n-1} a_{n-2}=\dots = b_n + D_1 b_{n-1} + D_{2} b_{n-2}+\dots +
D_n b_0\,,
$$
where $D_1:= d_n, D_2= d_n d_{n-1},\dots, |D_k| \le q^k$.
Therefore, $a=Ab$, where $A$ is a triangular matrix (with $1$'s one the diagonal), which clearly satisfies Schur test. Hence, $A:\ell^2\to \ell^2$ and lemma is proved.
\end{proof} 

In fact we need this type of lemma on the dyadic tree $T$ of dyadic lattice $\cD(I_0)$. But here it is.

\begin{lemma}
\label{SchurT}
Let $m(I, J)$ is a matrix on $T$ such that $m(I, J)= 0$ if $I\subsetneq J$ and for $c>0$
$$
m(I, J) =\Big(\frac{|I|}{|J|}\Big)^c
$$
otherwise. Then this matrix is a bounded operator from $\ell^2(T)$ to itself.
\end{lemma}
\begin{proof}
This matrix acts boundedly from $\ell^\infty(T)$ to $\ell^1(T)$. In fact, 
$$
\sum_{J} m(I, J) \le C_0 \quad \forall I\,.
$$
But its adjoint also acts boundedly from $\ell^\infty(T)$ to $\ell^1(T)$. In fact
$$
\sum_{I} m(I, J) \le C_0 \quad \forall J\,.
$$
By interpolation it maps $\ell^2(T)$ to itself boundedly.
\end{proof}

\section{Back to  the estimate from below of \eqref{ppagain}, \eqref{ppagainstar}}
\label{analyze}

First we just loosely discuss what we are going to do. The ``actual idea" follows in Subsection \ref{onemore}.
Disclaimer: Subsection \ref{onemore} contains three deliberate mistakes.  We will indicate all of them at the end of Subsection \ref{onemore}. If the reader can circumvent those problems, then the main result will be proved: repeated dyadic paraproduct is bounded if and only if its symbol is in Chang--Fefferman dyadic $BMO$.

\begin{remark}
\label{extra}
We actually prove this result in Theorem \ref{mainskip1} below, but only under the extra conditions on the symbol. We cannot get rid of those extra assumptions at this time.
\end{remark}

Let us repeat: we just need to be able to estimate the  sum in \eqref{ppagain} from below by constant or by $c\sum_{I\times J\subset U} (b, h_I\otimes h_J)^2$. Even the estimate from below by
$$
c\sum_{I\times J\subset U, \,I\times J \, is \,\,m\, deep, \, I, J\, even, even} (b, h_I\otimes h_J)^2
$$
would be a meaningful achievement. Similarly, even the estimate from below by
$$
c\sum_{I\times J\subset U, \,I\times J \, is \,\,m\, deep, \, I, J\, odd, odd} (b, h_I\otimes h_J)^2
$$
would be a meaningful achievement as well.

Recall
\begin{align}
\label{ppagain1}
\pi\pi(\al_0, \one_U; \bfT)= &\sum (\al_0, h_I\otimes h_J) (\one_U , \tilde\one_I\otimes \tilde\one_J) T_1 h_I\otimes T_2 h_J -
\\
&\sum (\al_0, h_I\otimes h_J) (\one_U , T_1^*\tilde\one_I\otimes \tilde\one_J)  h_I\otimes T_2 h_J -\notag
\\
&\sum (\al_0, h_I\otimes h_J) (\one_U , \tilde\one_I\otimes T_2^*\tilde\one_J) T_1 h_I\otimes  h_J +\notag
\\
&\sum (\al_0, h_I\otimes h_J) (\one_U , T_1^*\tilde\one_I\otimes T_2^*\tilde\one_J)  h_I\otimes h_J= I -II-III+IV\,.\notag
\end{align}

\begin{defin}
\label{eps1}
Let $0<\eps_1\le \frac1{100}$. The rectangle $I\times J\subset U$ is called $\eps_1$ small if 
$$
|(\one_U , T_1^*\tilde\one_I\otimes \tilde\one_J)|\le \eps_1\,.
$$
\end{defin}

\begin{lemma}
\label{smallup}
Let $I\times J$ be $m$-deep in $U$, $m\ge 2$. Let $R=I\times J$ be $\eps_1$-small. Then $R_1=\hat I\times J$ and $R_2 = I\times \hat J$, $\hat R=\hat I \times \hat J$ are $\eps_1$-small.
\end{lemma}
\begin{proof}
Let us consider $R_1$. We look at
$$
(\one_U , T_1^*\tilde\one_I\otimes \tilde\one_J) - (\one_U , T_1^*\tilde\one_{\hat I}\otimes \tilde\one_J)
$$
and notice that  $\tilde\one_I - \tilde\one_{\hat I}$ is proportional to $h_{\hat I}$, so 
$\tilde\one_I\otimes \tilde\one_J - \tilde\one_{\hat I}I\otimes \tilde\one_J$  is proportional to $h_{\hat I}\otimes h_J$.
Hence $T_1^* (\tilde\one_I\otimes \tilde\one_J - \tilde\one_{\hat I}I\otimes \tilde\one_J)$ is proportional to
$h_{\mathcal{I}}\otimes h_J$, where $\mathcal{I}$ is the parent of $\hat I$ (grandparent of $I$). The scalar product with $\one_U$ returns $0$ if $R$ is at least $2$ deep.

\bigskip

\begin{defin}
Let $\tilde\one_R$ temporarily denote $\tilde\one_I\otimes \tilde\one_J $ for $R= I\times J$.
\end{defin}

We proved that
\begin{equation}
\label{equ1}
(\one_U , T_1^* \tilde\one_R) = (\one_U , T_1^* \tilde\one_{R_1})\,,
\end{equation}
if $R$ is at least $2$-deep. Therefore, $R_1$ is also $\eps_1$ small.

\bigskip

The same proof works for $R_2$. And then going from $R_2$ to $\hat R$ involves the same reasoning.
In particular we get
\begin{equation}
\label{equ12}
(\one_U , T_1^*  \tilde\one_R) = (\one_U , T_1^*  \tilde\one_{R_1})= (\one_U , T_1^*  \tilde\one_{R_2})= (\one_U , T_1^*  \tilde\one_{\hat R})\,.
\end{equation}
And we conclude that all $R_1, R_2, \hat R$ are $\eps_1$ small as soon as $R$ is such.

\end{proof}

\begin{lemma}
\label{smalldown}
Let $R= I\times J$ be such that $R\subset U$ and moreover $R$ is at least $1$ deep. Let $R$ be $\eps_1$-small. Let $I'\in ch(I)$, $J'\in ch (J)$. Then $R^1=\hat I'\times J$ and $R^2 = I\times \hat J'$, $ R^{12}=\hat I' \times \hat J'$ are $\eps_1$-small.
\end{lemma}
\begin{proof}
The proof is exactly the same, again the difference of two characteristic functions normalized as we did gives us a proportional to Haar function, $T_1^*$ moves it only one step up, and so the result is still orthogonal to $\one_U$.
\end{proof}

Now we can prove (only with $I$ and $J$ interchanged) that the same \eqref{equ12} will hold for the terms $(\one_U , \tilde\one_I\otimes T_2^*\tilde\one_J)$ instead of $(\one_U ,  T_1^*\tilde\one_I\otimes\tilde\one_J)$. And then, again we can repeat the same reasoning for the terms
$(\one_U , T_1^*\tilde\one_I\otimes T_2^*\tilde\one_J)$.

\bigskip

Conclusion: 
\begin{itemize}
\item 1) in \eqref{ppagain} either all terms  $(\one_U ,  T_1^*\tilde\one_I\otimes\tilde\one_J)$, $(\one_U , \tilde\one_I\otimes T_2^*\tilde\one_J)$, $(\one_U , T_1^*\tilde\one_I\otimes T_2^*\tilde\one_J)$ are $\frac1{100}$ small, or
\item 2) two brand of terms, e.g. all $(\one_U ,  T_1^*\tilde\one_I\otimes\tilde\one_J)$, $(\one_U , \tilde\one_I\otimes T_2^*\tilde\one_J)$ are $\frac1{100}$  small and all $(\one_U , T_1^*\tilde\one_I\otimes T_2^*\tilde\one_J)$ are NOT $\frac1{100}$  small (there are three choices of such a situation), or
\item 3)  one brand of terms, e.g. all $(\one_U ,  T_1^*\tilde\one_I\otimes\tilde\one_J)$ are all  $\frac1{100}$  small, and two other brands are NOT $\frac1{100}$  small, or
\item 4) all terms of all three brands are NOT $\frac1{100}$  small.
\end{itemize}

Let us consider item 1) and item 4). The rest should follow suit.

\bigskip

So suppose all terms of all three brands are $\frac1{100}$  small. Then notice that $(b, h_I\otimes h_J)$ from the first line of \eqref{ppagain} has coefficient $(\one_U \tilde\one_I\otimes\tilde\one_J)=1$--and of course all this stands in front of some $h_{I'}\otimes h_{J'}$ (children). If we collect the same  $h_{I'}\otimes h_{J'}$  terms from all $3$ other lines of \eqref{ppagain} the $b$-Haar coefficient (maybe not $(b, h_I\otimes h_J)$, but some other Haar coefficient of a {\it some} rectangle ``close'' to $I\times J$) will appear but with the factor less or equal than $\frac1{100}$ by absolute value.  So we will have that the sum in \eqref{ppagain} is exactly equal to the sum of orthogonal terms
$$
[1\cdot (b, h_I\otimes h_J)+ \eps_1 (b, \dots) +\eps_2 (b, \dots) + \eps_3(b, \dots)] h_{I'}\otimes h_{J'}\,,
$$
where $\dots$ stand for some $h_{I''_i\times J''_i}$, $i=1,2,3$, where $I''_i\times J''_i$ is a close relative of $I\times J$.
Hence \eqref{ppagain} has the norm squared bounded from below by the sum
$$
|1\cdot (b, h_I\otimes h_J)+ \eps_1 (b, \dots) +\eps_2 (b, \dots) + \eps_3(b, \dots)|^2\,,
$$
where the summation is over all those $I\times J$, that are, say, at least $2$-deep in $U$ and $I, J$ are both even.
And of course $|\eps_i|\le 1/100$.

Then using the analog of Lemma \ref{SchurT} but for double tree, we conclude that the square of the norm of the sum in \eqref{ppagain} is bounded from below by
the sum
$$
c\,| (b, h_I\otimes h_J)|^2\,,
$$
where the summation is over all $I\times J$, say, at least $2$-deep in $U$ and  and $I, J$ are both even,
 and we got what we wanted at least for $even, even$ part of $b$.

\bigskip

Let us consider now item 4): all coefficients of all three brands are bigger than $\frac1{100}$ by absolute value.

Let us start with the fourth line of \eqref{ppagain}. It contains $I$ even, $J$ even sum, and this sum is orthogonal to the rest. So immediately we get (using that $|(\one_U , T_1^*\tilde\one_I\otimes T_2^*\tilde\one_J)|\ge 1/100$ by assumption 4))
\begin{equation}
\label{ee1}
\sum_{I even, \, J\, even} (b, h_I\otimes h_J)^2 \le C\|\text{sum of \eqref{ppagain}}\|_2^2\,.
\end{equation}

So looks like we estimated $\|b^{even, even}\|_{BMO_{ChF}}$ from above  by the norm of the repeated commutator.

\bigskip


\subsection{``A proof'' of the Chang--Fefferman BMO norm estimate of $b^{odd, odd}$}
\label{onemore}

In this subsection we prefer to work with $T_1^*, T_2^*$ rather than $T_1, T_2$ and we notice that this change is basically symmetric. So our goal will be to estimate  Chang--Fefferman norm of $b^{even, even}$ by the norm of
$[T_1^*, [T_2^*, b]]$.

Absolutely symmetrically to what has been done above we can reduce the above mentioned  task to estimating the norm
of the following sum from below by Chang--Fefferman norm of $b^{odd, odd}$.
\begin{align}
\label{ppagain1}
\pi\pi(\al_0, \one_U; \bfT^*)= &\sum (\al_0, h_I\otimes h_J) (\one_U , \tilde\one_I\otimes \tilde\one_J) T_1^* h_I\otimes T_2^* h_J -
\\
&\sum (\al_0, h_I\otimes h_J) (\one_U , T_1\tilde\one_I\otimes \tilde\one_J)  h_I\otimes T_2^* h_J -\notag
\\
&\sum (\al_0, h_I\otimes h_J) (\one_U , \tilde\one_I\otimes T_2\tilde\one_J) T_1^* h_I\otimes  h_J +\notag
\\
&\sum (\al_0, h_I\otimes h_J) (\one_U , T_1\tilde\one_I\otimes T_2\tilde\one_J)  h_I\otimes h_J= I -II-III+IV\,.\notag
\end{align}
The summation is over $I, J$, $I\times J\subset U$.

\begin{defin}
Let $0<\eps_1\le \frac1{100}$. The rectangle $I\times J\subset U$ is called $\eps_1$ small if 
$$
|(\one_U , T_1\tilde\one_I\otimes \tilde\one_J)|\le \eps_1\,.
$$
\end{defin}

Let us start with the fourth line of \eqref{ppagain1}. It contains $I$, $J$ (odd, odd) sum, and this sum is orthogonal to the rest. So immediately we get, using the assumption $|(\one_U , T_1\tilde\one_I\otimes T_2\tilde\one_J)|\ge \eps_2$, that
\begin{equation}
\label{oo1}
\sum_{I odd, \, J\, odd} (b, h_I\otimes h_J)^2 \le C\eps_2^{-2}\|\text{sum of \eqref{ppagain}}\|_2^2\,.
\end{equation}

The next lemma says that if $|(\one_U , T_1\tilde\one_I\otimes T_2\tilde\one_J)|\ge \eps_2$ happens for one $I\times J$ inside $U$ it happens for all such rectangles inside $U$. (If $U$  is connected, which we can assume by considering connected components one by one.)

We just saw that the assumption $|(\one_U , T_1\tilde\one_I\otimes T_2\tilde\one_J)|\ge \eps_2$ implies \eqref{oo1}.  Let us now assume that $|(\one_U , T_1\tilde\one_I\otimes T_2\tilde\one_J)|\le \eps_2$ for (odd, odd) $I, J$, then by lemmas below we can assume that for all $I\times J$ inside $U$ we have
$$
|(\one_U , T_1\tilde\one_I\otimes T_2\tilde\one_J)|\le \eps_2\,.
$$

\bigskip

\begin{defin}
Let $\tilde\one_R$ temporarily denote $\tilde\one_I\otimes \tilde\one_J $ for $R= I\times J$.
\end{defin}

\begin{lemma}
\label{smallup1}
Let $I\times J$ be $m$-deep in $U$, $m\ge 2$. Let $R=I\times J$ be $\eps_1$-small. Then $R_1=\hat I\times J$ and $R_2 = I\times \hat J$, $\hat R=\hat I \times \hat J$ are $\eps_1$-small.
\end{lemma}
\begin{proof}
Let us consider $R_1$. We look at
$$
(\one_U , T_1\tilde\one_I\otimes \tilde\one_J) - (\one_U , T_1\tilde\one_{\hat I}\otimes \tilde\one_J)
$$
and notice that  $\tilde\one_I - \tilde\one_{\hat I}$ is proportional to $h_{\hat I}$, so 
$\tilde\one_I\otimes \tilde\one_J - \tilde\one_{\hat I}I\otimes \tilde\one_J$  is proportional to $h_{\hat I}\otimes h_J$.
Hence $T_1 (\tilde\one_I\otimes \tilde\one_J - \tilde\one_{\hat I}I\otimes \tilde\one_J)$ is proportional to
$h_{\mathcal{I}}\otimes h_J$, where $\mathcal{I}$ is the parent of $\hat I$ (grandparent of $I$). The scalar product with $\one_U$ returns $0$ if $R$ is  in $U$.

We proved that
\begin{equation}
\label{equ1}
(\one_U , T_1 \tilde\one_R) = (\one_U , T_1 \tilde\one_{R_1})\,,
\end{equation}
if $R$ is in $U$. Therefore, $R_1$ is also $\eps_1$ small.

\bigskip

The same proof works for $R_2$. And then going from $R_2$ to $\hat R$ involves the same reasoning.
In particular we get
\begin{equation}
\label{equ12}
(\one_U , T_1  \tilde\one_R) = (\one_U , T_1  \tilde\one_{R_1})= (\one_U , T_1 \tilde\one_{R_2})= (\one_U , T_1 \tilde\one_{\hat R})\,.
\end{equation}
And we conclude that all $R_1, R_2, \hat R$ are $\eps_1$ small as soon as $R$ is such.

Notice that there was nothing special about $T_1$ in this equality. So we can write the same thing for the action by $T_2$ and for the action of $T_1\otimes T_2$:

\begin{equation}
\label{equ12T2}
(\one_U , T_2  \tilde\one_R) = (\one_U , T_2  \tilde\one_{R_1})= (\one_U , T_2 \tilde\one_{R_2})= (\one_U , T_2\tilde\one_{\hat R})\,.
\end{equation}

\begin{equation}
\label{equ12T12}
(\one_U , T_1\otimes T_2  \tilde\one_R) = (\one_U , T_1\otimes T_2  \tilde\one_{R_1})= (\one_U , T_1\otimes T_2 \tilde\one_{R_2})= (\one_U , T_1\otimes T_2\tilde\one_{\hat R})\,.
\end{equation}

\end{proof}

\begin{lemma}
\label{smalldown1}
Let $R= I\times J$ be such that $R\subset U$ and moreover $R$ is at least $1$ deep. Let $R$ be inside $U$, and let it be $\eps_1$-small. Let $I'\in ch(I)$, $J'\in ch (J)$. Then $R^1=I'\times J$ and $R^2 = I\times J'$, $ R^{12}= I' \times J'$ are $\eps_1$-small.
\end{lemma}
\begin{proof}
The proof is exactly the same, again the difference of two characteristic functions normalized as we did gives us a function proportional to Haar function, $T_1$ moves it only one step down, and so the result is still orthogonal to $\one_U$.
\end{proof}

Let us fix two small constants $\eps_2<<\eps_1 \le \frac1{100}$.  We look at the terms 
$(\one_U , T_1\tilde\one_I\otimes T_2\tilde\one_J) $ in \eqref{ppagain1} and we know by Lemma \ref{smallup1} and Lemma \ref{smalldown1} that all those terms are the same. So if they are   for all $I, J$ (even, even) are like this:
$$
|(\one_U , T_1\tilde\one_I\otimes T_2\tilde\one_J)| \ge \eps_2
$$
then \eqref{ee1} is immediately true with $C=\eps_2^{-2}$. Now suppose that all of those terms are like that:
$$
|(\one_U , T_1\tilde\one_I\otimes T_2\tilde\one_J)| \le \eps_2\,.
$$
Notice that then the same is true for $I, J$  (even, odd). In fact,  this follows from \eqref{equ12T12}. Those (even, odd) output terms also happens in this sum
$$
 \sum (\al_0, h_{I'}\otimes h_{J'}) (\one_U , \tilde\one_{I'}\otimes T_2\tilde\one_{J'})T_1^* h_{I'}\otimes  h_{J'},
 $$
 when $I', J'$ are (odd, odd). Notice that the sum itself is in (even, odd) space.
 Therefore, when we consider the (even, odd) output it  can come only from the fourth and the third line of \eqref{ppagain1}. In fact, the first line of \eqref{ppagain1} is in (even, even) space and the second line is totally in (odd, even) space. Therefore, each term of this (even, odd) part of the output  will have the form ($I$ odd, $J$ odd):
 
\begin{align*}
&h_I\otimes h_J [ (\al_0, h_{\hat I} \otimes h_J)(\text{smaller by abs. value than}\, \eps_2 )+ (\al_0, h_{ I} \otimes h_J)
\cdot (\one_U , \tilde\one_{ I}\otimes T_2\tilde\one_J)]\,.
\end{align*}
Hence, taking (even, odd) part of \eqref{ppagain1} we have
$$
\sum_{I odd, J odd} |(\al_0, h_{\hat I} \otimes h_J)(\text{smaller by abs. value than}\, \eps_2 )+ (\al_0, h_{\hat I} \otimes h_J) (\one_U , \tilde\one_{ I}\otimes T_2\tilde\one_J)|^2 \le C\|\text{sum in }\, \eqref{ppagain1}\|_2^2\,.
$$

Now we remember that   either 1) all $|(\one_U , \tilde\one_{\hat I}\otimes T_2\tilde\one_J)| \ge \eps_1>>\eps_2$; or 2)  
all $|(\one_U , \tilde\one_{\hat I}\otimes T_2\tilde\one_J)| \le \eps_1$.

\bigskip

Let us formulate now an analog of Lemma \ref{SchurT}, but for a dyadic bi-tree $\cT\times\cT$. Bi-tree is a graph of dyadic rectangles, where the edges connect rectangles iff one of them is inside the other and the ratio of their areas is $2$.  In the next lemma $R=I\times J, R'=I'\times J'$. In other words $R, R'$ are two end-points of the same edge if either $I$ is a parent of $I'$, $J'=J$, or $I'$ is a parent of $I$, $J'=J$; or $J$ is a parent of $J'$, $I'=I$, or $J'$ is a parent of $J$, $I'=I$.
\begin{lemma}
\label{SchurTT2}
Let us consider a matrix enumerated by the couples of vertices $(R, R')$,  $R'\in \cT\times\cT$, $R'\in \cT\times\cT$. Let its matrix elements
$m(R, R')=0$ unless $R\subset R'$ or $R'\subset R$. In the latter cases, if $R\subset R'$, then $m(R, R') =\Big(\frac{|R|}{|R'|}\Big)^{A}$, and  if $R'\subset R$, then $m(R, R') =\Big(\frac{|R'|}{|R|}\Big)^{A}$, where $A$ is a sufficiently large constant. Then this matrix defines a bounded operator from $\ell^2(\cT\times \cT)$ to itself.
\end{lemma}

\begin{proof}
This is easy to see by applying Schur's test as in the proof of Lemma \ref{SchurT}.
\end{proof}


\bigskip

If we have 1), then by  Lemma \ref{SchurTT2} (it is an analog of Lemma \ref{SchurT} for bi-tree) we will get from the last display
\begin{equation}
\label{oo2}
\sum_{I odd, J odd}  (\al_0, h_{ I} \otimes h_J)^2  \le C\eps_1^{-2}\|\text{sum in }\, \eqref{ppagain1}\|_2^2\,.
\end{equation}
This is because $\eps_1>>\eps_2$. And \eqref{oo2} is  exactly \eqref{oo1}.

\bigskip

At this moment we established the following dichotomy: either we have proved \eqref{oo1}, or, otherwise we  have $|(\one_U , \tilde\one_{ I}\otimes T_2\tilde\one_J)| \le \eps_1$ (for all $I, J$!), and we also have 
$|(\one_U , T_1\tilde\one_{ I}\otimes T_2\tilde\one_J)| \le \eps_2$ (again for all $I, J$!). When we say ``all'' we mean all $I, J$ such that $I\times J\subset U$. (We think  here that $U$ is connected, this does not restrict the generality as we can consider the connected components of $U$ one by one.)

 \bigskip
 
Now we call the attention of the reader that third and second lines of \eqref{ppagain1}  are totally symmetric under the interchange of $I$ and $J$. So we consider again two cases:  either 1) $|(\one_U , T_1\tilde\one_{ I}\otimes \tilde\one_J)| \le \eps_1$ (for all $I, J$); or, 2) $|(\one_U , T_1\tilde\one_{ I}\otimes \tilde\one_J)| \ge \eps_1$ (for all $I, J$).

In the second case we obtain a  form of \eqref{oo2} symmetric under $I$ versus $J$ exchange:

\begin{equation}
\label{oo3}
\sum_{I odd, J odd}  (\al_0, h_{ I} \otimes h_{ J})^2 \le C\eps_1^{-2}\|\text{sum in }\, \eqref{ppagain1}\|_2^2\,.
\end{equation}
But \eqref{oo3} is  exactly \eqref{oo1}.

\bigskip

Now we are left with case 1), namely, $|(\one_U , T_1\tilde\one_{ I}\otimes \tilde\one_J)| \le \eps_1$ (for all $I, J$).
We remind that we came to the following conclusion: either we already proved \eqref{oo1} or we have  two other sets of inequalities:  $|(\one_U , \tilde\one_{ I}\otimes T_2\tilde\one_J)| \le \eps_1$ (for all $I, J$!), and we also have 
$|(\one_U , T_1\tilde\one_{ I}\otimes T_2\tilde\one_J)| \le \eps_2$ (again for all $I, J$!). Again, when we say ``all'' we mean all $I, J$ such that $I\times J\subset U$.

\bigskip

So all three terms $|(\one_U , T_1\tilde\one_{ I}\otimes T_2\tilde\one_J)|$, $|(\one_U , \tilde\one_{ I}\otimes T_2\tilde\one_J)|$, $|(\one_U , T_1\tilde\one_{ I}\otimes T_2\tilde\one_J)| $  for all $I, J$ are small.

\bigskip

At this moment we consider the first line of \eqref{ppagain1}. Notice that its output is in the (even, even) space by the definition of $T_1^*, T_2^*$.

All other three lines also have some of their output in (even, even) space. So we collect the factor in front of  the same
$h_{I'}\otimes h_{J'}$, $I', J'$ being (even, even) from all four lines of \eqref{ppagain1}.

We will have that the (even, even) part of the  sum in \eqref{ppagain1} is exactly equal to the sum of orthogonal terms ($\hat I$, $\hat J$ are (even, even)):
$$
[1\cdot (b, h_I\otimes h_J)+ \tilde\eps_1 (b, \dots) +\tilde\eps_2 (b, \dots) +\tilde \eps_3(b, \dots)] h_{\hat I}\otimes h_{\hat J}\,,
$$
where $I, J$ are children of of $\hat I, \hat J$ correspondingly, where dots $\dots$ stand for some $h_{I''_i\times J''_i}$, $i=1,2,3$, and where $I''_i\times J''_i$ is a close relative of $I\times J$. These close relatives will be $h_{\hat I}\otimes h_J$, $h_I\otimes  h_{\hat J}$ and $h_{\hat I}\otimes h_{\hat J}$. 
Also  $|\tilde \eps_1|\le \eps_1<< \eps_2$, $|\tilde \eps_2|\le \eps_1<<\eps_2$, $|\tilde\eps_3|\le \eps_2$.

\bigskip

Hence \eqref{ppagain1} has the norm squared bounded from below by the sum
$$
\sum_{I, J\, are \,odd, odd\dots}|1\cdot (b, h_I\otimes h_J)+ \tilde\eps_1 (b, h_{\hat I}\otimes h_J) +\tilde\eps_2 (b, h_I\otimes  h_{\hat J}) + \tilde\eps_3(b, h_{\hat I}\otimes h_{\hat J})|^2\,,
$$
where the summation is over all those $I\times J$, that are inside  $U$ and $I, J$ are both odd.
As we have seen $|\tilde \eps_1|\le \eps_1<< \eps_2$, $|\tilde \eps_2|\le \eps_1<<\eps_2$, $|\tilde\eps_3|\le \eps_2$.
Using this smallness and  Lemma \ref{SchurTT2}  for double tree, we conclude that the square of the norm of the sum in \eqref{ppagain1} is bounded from below by
the sum
$$
c\,| (b, h_I\otimes h_J)|^2\,,
$$
where the summation is over all $I\times J$, say, at least $1$-deep in $U$ and  and $I, J$ are both odd,
 and we got what we wanted at least for $odd, odd$ part of $b$, namely we got \eqref{ee1}.
 
 We just proved the following ``theorem''.
 \begin{theorem}
 \label{oothm}
 Let $T_1, T_2$ be dyadic shifts introduced at the beginning of this article. There is a universal constant $C<\infty$ such that
 $$
 \|b^{odd, odd}\|_{BMO_{ChF}} \le C\,\|[T_1^*, [T_2^*, b]]\|\,.
 $$
 \end{theorem}
 
 \bigskip

Now if we have repeated commutator with not just one but with $4$ shifts of our type (our shift is even even to odd odd, but there are three analogous shifts obviously with different input/output). 

If repeated commutator of $b$ with all these $4$ shifts  are bounded we conclude as before that the full $\|b\|_{BMO_{ChF}}$ norm is the estimate from above for the maximum of norms of those repeated commutators with multiplication by $b$. To formulate the theorem, let us denote the shifts we used above by 
$$
T_1=: T^{eo}_1, \, T_2=: T^{eo}_2\,.
$$
Notice that those shifts have one drawback: they have huge kernels. So it is easy to understand that may be repeated commutator looses some information about the symbol $b$.

However we have three more couple $T^{oe}_1, T^{oe}_2$, $T^{eo}_1, T^{oe}_2$, $T^{oe}_1, T^{eo}_2$.
Then we can repeat the argument above for each of those couples. Then we get the ``theorem'':
\begin{theorem}
 \label{mainthm}
 There s a universal constant $C<\infty$ such that
 $$
 \|b\|_{BMO_{ChF}} \le C\,\max\big[\|[T^{eo}_1, [T^{eo}_2, b]]\|,\|[T^{oe}_1, [T^{oe}_2, b]]\|,\|[T^{eo}_1, [T^{oe}_2, b]]\|, \|[T^{oe}_1, [T^{eo}_2, b]]\| \big]\,.
 $$
 \end{theorem}
 
It is well known that  the estimates from above in this theorem also holds, see e.g. \cite{LM}.

\subsection{Discussion of the deliberate mistakes}
\label{disc}

\begin{enumerate}
\item In Lemmas \ref{smallup}, \ref{smallup1},  and equation \eqref{equ1} we checked the claim for $R_1=\hat I\times J$, and it is correct. But it is incorrect for  and $R_2 = I\times \hat J$. In other words it is not true that one small coefficients automatically means that all of them are small. We can ``spread smallness'' horizontally, but not vertically. Lemmas \ref{smallup}, \ref{smallup1} were about $T_1$. The similar claims about $T_2$ would allow to ``spread smallness'' vertically, but not horizontally. So we cannot ``spread smallness" throughout the whole of connected  set $U$.
\item In estimating the repeated paraproduct from below we can disregard all parts with $D$ (see Section \ref{blocks}), as they are bounded by $BMO_r$ norm of the symbol. This is true.  Also it is trues that substituting $f=\one_U$ eliminates all $ZZ$ sums. But it is not true that substituting $f=\one_U$ eliminates all  $\pi Z, Z\pi$ terms of repeated paraproduct and leaves us only with $\pi\pi$ part. In fact, if one take a look at all the $16$  first sums comprising  repeated paraproduct  except for parts with $D$ (see the beginning of Section \ref{skipping}), we can see that $\pi Z$ unfortunately does not cancel out completely. The term $\pi Z3$ contains $(T_1^*, \tilde \one_I\otimes h_J)$, and this term is not obliged to be zero. (But  $(T_2^*, \tilde \one_I\otimes h_J)$ is zero.) Again $\pi Z4$ contains $(T_1^*T_2^*, \tilde \one_I\otimes h_J)$, and this term is not obliged to be zero. So $\pi Z3, \pi Z4$ survive, which is very unfortunate, they join all  for lines of $\pi\pi$ survivors.
Moreover, by the same reason $Z\pi2$ and $Z\pi4$ also survive. The more terms survive (now we have $6$ of them, the more difficult is to separate their action.
\item At least all $\pi\pi$ terms are almost orthogonal. Adding $\pi Z3, \pi Z4$,  $Z\pi2$ and $Z\pi4$, makes the survival part not to be the sum of almost orthogonal terms.
\item Let us assume for a while that we are left only with almost orthogonal part, that is, with $\pi\pi$. This would happen if we consider the repeated commutator, but not with multiplication by $b$ but with bi-parameter paraproduct $\pi_b$. This is also interesting, in principle. Let us take a look how we can use almost orthogonality. So $f=\one_U, b=\al_0$ and we collect the terms in front of
$h_K\otimes h_L$ in $\pi\pi$ all four lines below. Here is the coefficient in front of $h_K\otimes h_L$. Recall that $K, L$ are such that $\hat K\times \hat L\subset U$.
\begin{align*}
&\sign(K, \hat K)\sign(L, \hat L) (b, h_{\hat K} \otimes h_{\hat L}) \cdot (\one_U, \tilde \one_{\hat K}\otimes\tilde\one_{\hat L}) -
\\
&-\sign(K, \hat K) (b, h_{\hat K}\otimes h_L) (\one_U, \tilde\one_{\hat K}\otimes T_2 \tilde\one_L) -
\\
&-\sign(L, \hat L) (b, h_{ K}\otimes h_{\hat L}) (\one_U, T_1\tilde\one_{ K}\otimes  \tilde\one_{\hat L}) +
\\
& (b, h_{ K}\otimes h_L) (\one_U, T_1\tilde\one_{ K}\otimes T_2 \tilde\one_L) = 
\\
&\pm (b, h_{\hat K} \otimes h_{\hat L})  -
\\
&\pm (b, h_{\hat K}\otimes h_L) (\one_U, \tilde\one_{\hat K}\otimes T_2 \tilde\one_L) -
\\
&\pm (b, h_{ K}\otimes h_{\hat L}) (\one_U, T_1\tilde\one_{ K}\otimes  \tilde\one_{\hat L}) 
\\
& +(b, h_{ K}\otimes h_L) (\one_U, T_1\tilde\one_{ K}\otimes T_2 \tilde\one_L)
\\
=: x_{K, L}\,.
\end{align*}
So we are able to conclude immediately that the boundedness of the repeated commutator with bi-parameter $\pi_b$ (not with multiplication by $b$!) implies that
\begin{equation}
\label{xKL}
\sum_{K, L: \hat K\times \hat L\subset U} x_{K, L}^2 \le \|\text{repeated commutator with}\, \pi_b\|^2 |U|\,.
\end{equation}

But the ultimate goal would be to get a {\it different} inequality:
\begin{equation}
\label{xKL}
\sum_{K, L: \hat K\times \hat L\subset U} (b, h_{\hat K} \otimes h_{\hat L})^2 \le \|\text{repeated commutator with}\, \pi_b\|^2 |U|\,.
\end{equation}

Those inequalities are close to each other: just  by the previous calculation of $x_{K, L}$ above. But they are not the same. How to get rid of extra terms in $x_{K,L}$? Lemma \ref{SchurTT2} should help, but to apply it we need the smallness of ``parasite'' terms
$(\one_U, \tilde\one_{\hat K}\otimes T_2 \tilde\one_L)$, $(\one_U, T_1\tilde\one_{ K}\otimes  \tilde\one_{\hat L})$,
$(\one_U, T_1\tilde\one_{ K}\otimes T_2 \tilde\one_L)$. But the size of those terms is probably approximately $1$, and as we already noticed Lemma \ref{smallup1}, that dealt with this size, is only partially true.
\item And a final difficulty is that in $(b, h_{\hat K} \otimes h_{\hat L}) $ above we have only $\hat K,\hat  L$ (even, even) by the definition of our dyadic shifts.

\end{enumerate}


\section{The lists of building blocks of repeated commutators with $\bfT$ and $\bfT^*$}
\label{blocks}

Here are all the blocks of repeated commutator with $T_1^*\otimes T_2^*$. The blocks of repeated commutator with $T_1\otimes T_2$ are exactly the same, only stars are deleted.  All parts with $D$ letter are estimated by $\lesssim \|b\|_{BMO_r}$. Under the localization to $U_0$ we put $f=\one_U$, where $U$ is the enlargement of $U_0$. 

Then $ZZ$ part vanishes and $Z\pi 3$. The rest stays and should be estimated from below. In the next section we show that this is easy if $b$ has some extra property. In this section we show how difficult is the problem when this extra property is not assumed. However, there is a hope because the main part that is left after localization is $\pi\pi$ and it  has the property of almost orthogonality.

\bigskip

\begin{align*}
\pi\pi: & \sum(b, h_I\otimes h_J) \La f\Ra_{I\times J} (T_1^* h_I)\otimes (T_2^* h_J)
\\
- & \sum(b, h_I\otimes h_J) \La T_2^*f\Ra_{I\times J} (T_1^* h_I)\otimes ( h_J)
\\
- & \sum(b, h_I\otimes h_J) \La T_1^*f\Ra_{I\times J} (h_I)\otimes (T_2^* h_J)
\\
+& \sum(b, h_I\otimes h_J) \La T_1^* T_2^*f\Ra_{I\times J} (h_I)\otimes (h_J)
\end{align*}

\begin{align*}
\pi Z: & \sum(b, h_I\otimes h_J) ( f, \tilde\one_I \otimes h_J) (T_1^* h_I)\otimes (T_2^* \tilde\one_J)
\\
- & \sum(b, h_I\otimes h_J) ( T_2^*f, \tilde\one_I \otimes h_J) (T_1^* h_I)\otimes ( \tilde\one_J)
\\
- & \sum(b, h_I\otimes h_J) ( f, T_1^*\tilde\one_I \otimes h_J) ( h_I)\otimes (T_2^* \tilde\one_J)
\\
+& \sum(b, h_I\otimes h_J) ( T_1^* T_2^*f, \tilde\one_I \otimes h_J) ( h_I)\otimes ( \tilde\one_J)
\end{align*}

\begin{align*}
Z\pi: &  \sum(b, h_I\otimes h_J) ( f,  h_I \otimes \tilde\one_J)(T_1^*\tilde\one_I) \otimes (T_2^*h_J)
\\
-& \sum(b, h_I\otimes h_J) ( T_2^*f,  h_I \otimes \tilde\one_J)(T_1^*\tilde\one_I) \otimes (h_J)
\\
-&  \sum(b, h_I\otimes h_J) ( T_1^*f,  h_I \otimes \tilde\one_J)(\tilde\one_I )\otimes (T_2^*h_J)
\\
+ &  \sum(b, h_I\otimes h_J) (T_1^* T_2^*f,  h_I \otimes \tilde\one_J)(\tilde\one_I) \otimes (h_J)
\end{align*}

\begin{align*}
ZZ: &\sum(b, h_I\otimes h_J) ( f,  h_I \otimes h_J)(T_1^*\tilde\one_I) \otimes (T_2^*\tilde\one_J)
\\
- &\sum(b, h_I\otimes h_J) ( f,  T_2^*h_I \otimes h_J)(T_1^*\tilde\one_I) \otimes (\tilde\one_J)
\\
-&\sum(b, h_I\otimes h_J) (T_1^* f,  h_I \otimes h_J)(\tilde\one_I) \otimes (T_2^*\tilde\one_J)
\\
+&\sum(b, h_I\otimes h_J) (T_1^*T_2^* f,  h_I \otimes h_J)(\tilde\one_I) \otimes (\tilde\one_J)
\end{align*}

\begin{align*}
\pi D: -&\sum(b, h_I\otimes h_{\hat J}) (f, \tilde\one_I\otimes h_J) \frac1{\sqrt{
|\hat J|}} (T_1^* h_I)\otimes h_{\hat J} 
\\
+ &\sum(b, h_I\otimes h_{\hat J}) (T_1^*f, \tilde\one_I\otimes h_J) \frac1{\sqrt{
|\hat J|}} ( h_I)\otimes h_{\hat J} 
\end{align*}

\begin{align*}
Z D: -&\sum(b, h_{I}\otimes h_{\hat J}) (f, h_I\otimes h_J) \frac1{\sqrt{
|\hat J|}} (T_1^* \tilde\one_I)\otimes h_{\hat J} 
\\
+ &\sum(b, h_{ I}\otimes h_{\hat J}) (T_1^*f, h_I\otimes h_J) \frac1{\sqrt{
|\hat J|}} ( \tilde\one_I)\otimes h_{\hat J} 
\end{align*}

\begin{align*}
D\pi: -&\sum(b, h_{\hat I}\otimes h_J) (f, h_I\otimes \tilde\one_J) \frac1{\sqrt{
|\hat I|}} (h_{\hat I})\otimes T_2^*h_{ J} 
\\
+ &\sum(b, h_{\hat I}\otimes h_J) (T_2^*f, h_I\otimes \tilde\one_J) \frac1{\sqrt{
|\hat I|}} (h_{\hat I})\otimes h_{ J} 
\end{align*}

\begin{align*}
DZ: -&\sum(b, h_{\hat I}\otimes h_{ J}) (f, h_I\otimes h_J) \frac1{\sqrt{
|\hat I|}}  (h_{\hat I} )\otimes  (T_2^*\tilde\one_J)
\\
+ &\sum(b, h_{ \hat I}\otimes h_{J}) (T_2^*f, h_I\otimes h_J) \frac1{\sqrt{
|\hat I|}} (h_{\hat I})\otimes (\tilde \one_J)
\end{align*}

$$
DD: \sum(b, h_{\hat I}\otimes h_{ J}) (f, h_I\otimes h_J)h_{\hat I}\otimes h_{\hat J}
$$

\section{Why everything works for $b$ skipping scales?}
\label{skipping}

As nothing works for general $b$, let us assume that our $b$ skips scales. We mean here extra assumption that
\begin{equation}
\label{skip1}
(b, h_I\otimes h_J)\neq 0\Rightarrow I\, J, \text{are even, even}\,.
\end{equation}

In this special case we get the following
\begin{theorem}
\label{mainskip}
Let $b$ satisfy \eqref{skip1}. Then for any dyadic open set $U_0$ we have
$$
\sum_{I\times J\subset U_0} (b, h_I\otimes h_J)^2 \le C_1 \|\cT^b\|^2|U_0| + C_2 \|b\|_{BMO_{r}}^2|U_0|\,.
$$
\end{theorem}

\begin{proof}
Just in front  of this section the reader can see the table of all parts of the commutator. It is easy to  prove that
all parts $D\pi, \pi D, ZD, DZ$ are bounded by $\|b\|_{BMO_{r}}^2|U_0$. The $DD$ part is bounded just be
$$
\sup_{I, J} \frac{(b, h_I\otimes h_J)^2}{|I||J|}\|f\|_2^2,
$$
which is bounded by $\|b\|_{BMO_{r}}^2\|f\|_2^2$ again.

\bigskip

If $b$ is reduced to $\al_0$--the part of $b$ with Haar spectrum in $U_0$. We can  consider such reduced $b=\al_0$ by Lemma \ref{smsy}. At the same times this means that $I, J$ in $\pi\pi$, $\pi Z$, $Z\pi$, $ZZ$ parts are such that 
$I\times J\subset U_0$.

If $f= \one_U$, where $U$ is chosen in \eqref{UU0}, $|U|\le C|U_0|$, then $ZZ$ part vanishes,
if $b$ is reduced to $\al_0$--the part of $b$ with Haar spectrum in $U_0$. 

If $f= \one_U$, in $\pi nZ$ part the first and the second lines vanish, but $\pi Z3, \pi Z4$ do not vanish,
in $Z\pi$ part the first and the third lines vanish but $Z\pi2, Z\pi4$ do not vanish.

Notice that by assumption \eqref{skip1} $\pi Z3, \pi Z4$ are in (even, any), and $Z\pi2, Z\pi4$ are  in (any, even).

\bigskip

Let us take a look at the $\pi\pi1$-the first line of $\pi\pi$ part. It is is (odd, odd) part, and so it is orthogonal to 
$\pi Z3, \pi Z4$, $Z\pi2, Z\pi4$. 

It is also orthogonal to $\pi\pi2$, $\pi\pi3$, $\pi\pi4$-the second, third and fourth lines of $\pi\pi$ part.
In fact, by assumption \eqref{skip1} $\pi\pi2$ is in (odd, even), $\pi\pi3$ is in (even, odd), $\pi\pi2$ is in (even, even) parts.

Therefore,
\begin{equation}
\label{pZZp}
\sum_{I\times J\subset U_0} (\al_0, h_I\otimes h_J)^2 \le \|\pi\pi1(\al_0, \one_U)\|_2^2 \le \|\pi\pi(\al_0, \one_U)+\pi Z(\al_0, \one_U)+Z\pi(\al_0, \one_U)\|_2^2\,.
\end{equation}

On the other hand,
$$
\sup_{\|\rho\|_2=1, \sigma_H(\rho)\subset U_0} |(\one_U, \cT^b_*(\rho))|= \sup_{\|\rho\|_2=1, \sigma_H(\rho)\subset U_0} |(\one_U, \cT^{\al_0}_*(\rho))| =
$$
$$
\|\pi\pi(\al_0, \one_U)+\pi Z(\al_0, \one_U)+Z\pi(\al_0, \one_U) + ZZ(\al_0, \one_U) + \text{parts with D} (\al_0, \one_U)\|_2 \ge 
$$
$$
\|\pi\pi(\al_0, \one_U)+\pi Z(\al_0, \one_U)+Z\pi(\al_0, \one_U) \|_2 - \| \text{parts with D} (\one_U, \al_0)\|_2 \,.
$$

Now we already noticed that all parts $D\pi, \pi D, ZD, DZ, DD$ are bounded by $\|b\|_{BMO_{r}}^2|U_0$, and $ZZ(\al_0, \one_U)=0$. This and \eqref{pZZp} prove theorem.

\end{proof}

\begin{theorem}
\label{mainskip1}
Let $b$ satisfy \eqref{skip1}. Then  for any dyadic open set $U_0$
$$
\sum_{I\times J\subset U_0} (b, h_I\otimes h_J)^2 \le C_1 \|\cT^b\|^2|U_0| \,.
$$
\end{theorem}

\begin{proof}
Let us consider two cases: 1) $\|b\|_{BMO_r} \le \frac1{4C_2} \|b\|_{BMO_{CHF}}$, 2) $\|b\|_{BMO_r} \ge \frac1{4C_2} \|b\|_{BMO_{CHF}}$.
In the second case we use Section \ref{BMOr} and conclude the claim of the theorem.

In the first case we use Theorem \ref{mainskip}. We choose the dyadic open set  $U_0$ that gives us the Chang--Fefferman $BMO$ norm of $b$, and we write down the conclusion of Theorem \ref{mainskip}.
The left hand side of
$$
\sum_{I\times J\subset U_0} (b, h_I\otimes h_J)^2 \le C_1 \|\cT^b\|^2|U_0| + C_2 \|b\|_{BMO_{r}}^2|U_0|
$$
is $\|b\|_{BMO_{ChF}}^2|U_0|$ because of extremal property of $U_0$.  Now we use $\|b\|_{BMO_r} \le \frac1{4C_2} \|b\|_{BMO_{CHF}}$ to absorb $C_2 \|b\|_{BMO_{r}}^2|U_0|$ into the left hand side.
It is well known that  the estimates from above in this theorem also holds, see e.g. \cite{LM}.

\end{proof}

\section{A counterexample}
\label{firstex}

Many  results 
suffer from the same gap as \cite{FL}, if they all rely on the claim of the form that follows or a similar claim:
\begin{equation}
\label{no}
\|\phi\|_2 \lesssim \|P\phi\|_2\,,
\end{equation}
where $P$ is a projection given by a Fourier multiplier whose symbol is the indicator
of a product of half-spaces (in the original case of \cite{FL}, simply $\bR_+\times \bR_+\subset \bR^2$) and
the Fourier transform of $\phi$ is “symmetric”. In practice, $\phi$ takes the form $\phi=|\alpha|^2$
for some function $\alpha$ analytic in $\bC_+^2$. Obviously  $\phi$ is non-negative and in particular real-valued, so some symmetry of Fourier transform should follow.

\bigskip

In the one parameter case the simple symmetry shows that the previous claim is obviously correct. But not so in two parameter case anymore, as an example  below  shows.

The statement \eqref{no} should be avoided (in this generality $\phi=|\alpha|^2$, $\alpha $ being nice analytic function in $\bC_+^2$) in the future corrected proof. However, something like that should be in the proof, as it is the only estimate from below in the whole reasoning: all other estimates in \cite{FL} are the estimates of operators from above.

\bigskip

Let $\alpha$ be an analytic function in $\bC_+\times \bC_+$ and let it be just Fourier transform of $\psi(\xi), \xi=(\xi_1, \xi_2)\in \bR^2_{++}$, $\psi\in L^2(d\xi_1 d\xi_2)$.
Let the support of $\psi$ be in $Q:=[0, 2]^2$. In what follows $d\xi=d\xi_1 d\xi_2$.

$$
|\alpha(x_1, x_2)|^2 =  \int\int e^{i x_1(\xi_1-\eta_1) + i x_2 (\xi_2-\eta_2)} \psi(\xi) \bar \psi(\eta) d\xi d\eta=
$$
$$
\int\int e^{i x_1t_1 + i x_2 t_2} \psi(\xi_1, \xi_2) \overline{ \psi(\xi_1-t_1, \xi_2-t_2)} d\xi_1 d\xi_2 dt_1 dt_2
$$

Hence  Fourier transform 
$$
(\cF |\alpha|^2)(t_1, t_2) = \int_{Q_0} \psi(\xi_1, \xi_2) \overline{\psi(\xi_1-t_1, \xi_2-t_2)} d\xi_1d\xi_2
$$

If $\psi$ is a characteristic function of $\epsilon$ neighborhood of $\xi_1+\xi_2 =1, \xi_i\ge 0$, then  for positive vectors $t$ this quantity is uniformly bounded by 
$C\epsilon$. Moreover, this quantity is zero  if positive vector $t$ has norm bigger than $C\epsilon$.

Hence
$$
\|P_{++}|\alpha|^2\|_2 \asymp \sqrt{\epsilon^2\cdot \epsilon^2} = \epsilon^2
$$

On the other hand,
$$
\| \alpha\|_2 \asymp \epsilon^{1/2}\,.
$$

Also let us look at $P_{+-} |\alpha|^2$. Now we look at 
$$
(\cF |\alpha|^2)(t_1, -t_2) = \int_{Q_0} \psi(\xi_1, \xi_2) \overline{\psi(\xi_1-t_1, \xi_2+t_2)} d\xi_1d\xi_2
$$
for positive vector $t$. We can see that this is of order $\epsilon$ for positive vectors $t$ such that they fill-in the domain inside angle of size $\epsilon$ and diameter $1$ and of total measure of order $\epsilon$. 
Thus 
$$
\|P_{+-} |\alpha|^2\|_2^2 \asymp \epsilon^3\,.
$$
So
$$
\| \alpha\|_4^2 =\| |\alpha|^2\|_2 \ge \epsilon^{3/2}\,.
$$
In fact,
$$
\| \alpha\|_4^2 =\| |\alpha|^2\|_2 \asymp \epsilon^{3/2}\,.
$$

So both inequalities below are false:
$$
\|P_{++}|\alpha|^2\|_2 \ge c \| \alpha\|_2^2; \quad  \|P_{++}|\alpha|^2\|_2 \ge c \| \alpha\|_4^2.
$$

Normalization of $\alpha$ does not help. Let $\tilde \alpha = \epsilon^{-1/2} \alpha$. Then 
$$
\|\tilde \alpha\|_2\asymp 1,
$$
$$
\| \tilde \alpha\|_4^2 =\| |\tilde\alpha|^2\|_2 \asymp \epsilon^{1/2}\,.
$$
And 
$$
\|P_{++}|\tilde\alpha|^2\|_2 \asymp  \epsilon\,.
$$

So the inequality below is false:
$$
\|P_{++}|\tilde\alpha|^2\|_2 \ge c.
$$
Even though $\|\tilde \alpha\|_2\asymp 1$.

Thus the inequality on page 208 line 4 top of \cite{L} cannot be true as written. Similarly, inequalities  in \cite{FL} on page 149 line 3 at the end of the page cannot be true as written.

Symmetry works in estimating from below positive frequency part of the Fourier (of real function) via the whole Fourier transform of this function perfectly well for functions of one variable. But for functions of two variables this does not work. Even if the function is a) positive and b) the modulus squared of analytic function in 
$\bC_+\times \bC_+$  as the example above shows.

One can guess what is missing is one more normalization (which we do not have in the example above):
\begin{equation}
\label{normali}
1\asymp \|\tilde \alpha\|_2 \lesssim \|\tilde \alpha\|_4\,.
\end{equation}

However, there is a modification of this counterexample suggested by T. Hyt\"onen that has this property as well.


\end{document}